\documentclass[11pt, a4paper, UKenglish]{amsart}
\usepackage{geometry}
\usepackage[UKenglish]{babel}
\usepackage[utf8]{inputenc}
\usepackage{csquotes}
\usepackage{tikz-cd}
\usepackage{float} 
\usepackage{comment}
\usepackage{mathabx}
\usepackage{amssymb,amsmath,mathtools,amsthm}
\usepackage[toc,page]{appendix}
\usepackage{enumitem}
\usepackage{hyperref}
\usepackage{graphicx}
\usepackage{setspace} 
\usepackage{indentfirst}
\usepackage{breakcites}
\hypersetup{
	colorlinks=true,
	linkcolor=black,   
	citecolor=black,   
	urlcolor=black     
}

\usepackage[url=false,doi=false,isbn=false,backend=biber,style=alphabetic]{biblatex}

\theoremstyle{definition}
\newtheorem{theorem}{Theorem}[section]
\newtheorem{example}[theorem]{Example}
\newtheorem{definition}[theorem]{Definition}
\newtheorem{remark}[theorem]{Remark}
\newtheorem{lemma}[theorem]{Lemma}
\newtheorem{proposition}[theorem]{Proposition}
\newtheorem{question}[theorem]{Question}

\newtheorem*{ack}{Acknowledgements}
\newtheorem*{notations}{Notations and Conventions}

\allowdisplaybreaks
\newcommand{\BA}{{\mathbb{A}}}

\newcommand{\BC}{{\mathbb{C}}}

\newcommand{\BP}{{\mathbb{P}}}
\newcommand{\BQ}{{\mathbb{Q}}}

\newcommand{\BZ}{{\mathbb{Z}}}

\newcommand{\CA}{{\mathcal A}}

\newcommand{\CD}{{\mathcal D}}

\newcommand{\CF}{{\mathcal F}}
\newcommand{\CG}{{\mathcal G}}
\newcommand{\CH}{{\mathcal H}}
\newcommand{\CI}{{\mathcal I}}

\newcommand{\CL}{{\mathcal L}}

\newcommand{\CO}{{\mathcal O}}

\newcommand{\CR}{{\mathcal R}}

\newcommand{\CZ}{{\mathcal Z}}

\newcommand{\BLP}{\text{Bl}_{\BP^{2}\times 0 \bigcup 0\times \BP^{2}}(\BP^{5})}
\newcommand{\Uxyy}{U_{x_{0},y_{0},y_{3}}}
\newcommand{\Uxxyy}{U^{\prime}_{x_{0},y_{0},y_{3}}}

\usepackage[UKenglish]{babel}

\title{Equivariant Kuznetsov Components of Certain Cubic Fourfolds} 

\author{Xianyu Hu}

\address{Fakult\"at f\"ur Mathematik,
Technische Universit\"at M\"unchen, D-85747 Garching bei M\"unchen, Germany}
\email{xianyu.hu@tum.de}
\thanks{The author is supported by the Deutsche
Forschungsgemeinschaft (DFG, German Research Foundation) -- 516701553.}
\date{December 28, 2023}

\begin{document}
\begin{abstract}
Let $M$ denote a specific cubic fourfold that accommodates a group action by $\mathbb{Z}/3\mathbb{Z}$.
Through utilization of derived Mckay correspondence, we present a new proof establishing the identification of the equivariant Kuznetsov component in the equivariant derived category of $M$ with the derived category of certain abelian surface. This surface naturally emerges from the defining equation of the cubic fourfold $M$.
\end{abstract}
\maketitle

\setcounter{tocdepth}{1} 
\tableofcontents
\section*{Introduction}
The Kuznetsov component $\mathcal{K}u_{M}$ of a cubic fourfold $M$, introduced in \cite{kuznetsov2010derived}, is a full triangulated category
\begin{equation*}
\mathcal{K}u_{M}\coloneqq \langle \CO_{M},\CO_{M}(1),\CO_{M}(2)\rangle^{\bot}\subset D^{b}(M)
\end{equation*}
defined as the complement of three line bundles. It turns out to be a subtler and more interesting derived invariant of $M$ than the whole derived category. 

For instance, the Kuznetsov component $\mathcal{K}u_{M}$ behaves in many ways like the derived category $D^{b}(S)$ of a K3 surface $S$. In \cite{kuznetsov2010derived}, the author proposed a seminal conjecture connecting  properties of $\mathcal{K}u_{M}$ to the rationality problem of the cubic fourfold $M$. Precisely, he conjectured that the cubic fourfold $M$ is rational if and only if its Kuznetsov component $\mathcal{K}u_{M}$ is equivalent to the derived category of a K3 surface. 

Assume the cubic fourfold $M$ admits a group action by a finite group $G$. Then the line bundle $\CO_{M}(1)$ and the semiorthogonal decomposition 
\begin{equation*}
D^{b}(M)=\langle \mathcal{K}u_{M},\CO_{M},\CO_{M}(1),\CO_{M}(2)\rangle
\end{equation*}
are preserved by the group action of $G$. Hence we obtain the semiorthogonal decomposition 
\begin{equation*}
D^{b}_{G}(M)=\langle \mathcal{K}u_{M}^{G},\langle\CO_{M}\rangle^{G},\langle\CO_{M}(1)\rangle^{G},\langle\CO_{M}(2)\rangle^{G}\rangle
\end{equation*}
of the equivariant derived category of $M$, where $\langle\CO_{M}(i)\rangle^{G}$ denotes the equivariant category of the subcategory $\langle\CO_{M}(i)\rangle$ for $i=0,1,2$ and 
\begin{equation*}
\mathcal{K}u_{M}^{G}\coloneqq \langle \langle\CO_{M}\rangle^{G},\langle\CO_{M}(1)\rangle^{G},\langle\CO_{M}(2)\rangle^{G} \rangle^{\bot}    
\end{equation*}
is called the equivariant Kuznetsov component of $M$. It is natural to ask whether $\mathcal{K}u_{M}^{G}$ is equivalent to the derived category of a smooth variety and whether explicit examples satisfying this property can be constructed.
 
We study such an example. Let $M:= V(F)\subset \BP^{5}$ be a smooth cubic fourfold, where $F=\,F_{0}(x_{0},x_{1},x_{2})+\,F_{1}(x_{3},x_{4},x_{5})$ for $F_{0}$, $F_{1}$ cubic equations and $G:=\,\BZ/3\BZ$. We pick a generator $g$ of $G$, whose action on $M$ is given by
\begin{center}
    $g_{.}[x_{0}:x_{1}:x_{2}:x_{3}:x_{4}:x_{5}]=[x_{0}:x_{1}:x_{2}:e^{2\pi i/3}x_{3}:e^{2\pi i/3}x_{4}:e^{2\pi i/3}x_{5}]$.
\end{center}
The quotient space $M/G$ is denoted by $X$.
For simplicity, we denote $V(F_{0})$ by $E_{0}$ and $V(F_{1})$ by $E_{1}$, both being smooth elliptic curves inside $\BP^{2}$.

The main result of this paper is
\begin{theorem}(Theorem \ref{main theorem})
$\mathcal{K}u_{M}^{G}\cong D^{b}(E_{0}\times E_{1})$.
\end{theorem}
This has already been proved in \cite[Remark 5.8]{https://doi.org/10.48550/arxiv.1809.09940}, \cite[Main Theorem]{LIM2021103944}, \cite[Example 3.10]{ballard2014category} and \cite[Example 7.4]{beckmannOberdieck2020equivariant}. However, each of them used complicated methods and proved the result in different generalized forms. In this paper, we give a simpler proof in our particular case.

The main ingredient of our proof is the derived Mckay correspondence established in \cite{0f6d937083584c2ba1ca157da0f5ba0a} by Bridgeland, King and Reid. We use the $G$-Hilbert scheme $Y\coloneqq G\text{-Hilb}_{\BC}(M)$ to establish the following equivalence (see Theorem \ref{Theorem 3}):
\begin{theorem}
$D^{b}(Y)\cong D^{b}_{G}(M)$.
\end{theorem}
The strategy is as follows. Take a (universal) family $\CZ\subset Y\times M$ together with natural projections
\begin{equation*}
Y\xleftarrow[]{q} \CZ \xrightarrow[]{p} M.  
\end{equation*}
According to \cite{0f6d937083584c2ba1ca157da0f5ba0a}, if the fiber product 
\begin{center}
    $Y\times_{X}Y\,=\,\left\{(y_{1},y_{2})\in Y\times Y\,|\,\tau(y_{1})=\,\tau(y_{2})\right\}\subset Y\times Y$
\end{center}
has dimension $\leq 5$, then the functor
\begin{center}
    $\Phi:=\,\textbf{R}q_{*}\circ p^{*} \colon \,D^{b}(Y)\longrightarrow D^{b}_{G}(M)$
\end{center}
induces an equivalence. By analyzing the geometry of the $G$-Hilbert scheme $Y$, we have the following description of $Y$ (see Theorem \ref{Theorem 2}):
\begin{theorem}
Let $M^{G}$ denotes the fixed locus of the $G$-action on $M$. Then
\begin{equation*}
 G\text{-Hilb}_{\BC}(M)\cong (\text{Bl}_{M^{G}}(M))/G\cong \text{Bl}_{E_{0}\times E_{1}}(\BP^{2}\times \BP^{2}).
\end{equation*}
  \end{theorem}
Then we verify that $\text{dim } Y\times_{X}Y\leq 5$ and obtain the equivalence. We point out that our method to compute $G\text{-Hilbert schemes}$ can be applied to prove categorical Torelli theorem via equivariant Kuznetsov component of the non-Eckhardt type cubic threefold associated with involution, see \cite{Casalaina-Martin:2023}.

Our next step is to identify $\mathcal{K}u_{M}^{G}$ with $D^{b}(E_{0}\times E_{1})$ via an explicit functor. Let $\CO_{Y}(a,b,c)$ be the line bundle associated to the divisor $aE^{\prime}+bH_{1}+cH_{2}$, where $E^{\prime}$ is the exceptional divisor of the blow-up $\text{Bl}_{E_{0}\times E_{1}}(\BP^{2}\times \BP^{2})$, and $H_{i}$ denotes the divisor coming from the $i$-th factor of $\BP^{2}\times \BP^{2}$ for $i=1,\,2$. We explain in Section \ref{Section 5.2} that 
$D^{b}(Y)$ and $D^{b}_{G}(M)$ admit the following semiorthogonal decompositions 
 \begin{align*}
 D^{b}(Y)= \langle & D^{b}(E_{0}\times E_{1}),\CO_{Y}(0,0,0),\CO_{Y}(0,0,1),
       \CO_{Y}(0,0,2),
       \CO_{Y}(0,1,0),\\
      &\CO_{Y}(0,1,1),\CO_{Y}(0,1,2),\CO_{Y}(0,2,0),\CO_{Y}(0,2, 1),\CO_{Y}(0,2,2)\rangle,\\
 D^{b}_{G}(M)=  \langle & \mathcal{K}u_{M}^{G}, \CO_{M}\otimes \chi_{0}, \CO_{M}\otimes \chi_{1},
    \CO_{M}\otimes \chi_{2}, \CO_{M}(1)\otimes \chi_{0},
     \CO_{M}(1)\otimes \chi_{1},\\
     & \CO_{M}(1)\otimes \chi_{2}, \CO_{M}(2)\otimes \chi_{0},\CO_{M}(2)\otimes \chi_{1},\CO_{M}(2)\otimes \chi_{2} \rangle,  
 \end{align*}
where $\chi_{0},\chi_{1},\chi_{2}$ are characters of the group $G$ and $\CO_{M}(i)\otimes \chi_{0}$ is endowed with the canonical $G$-structure.

Naively, one may try to compare the image of $D^{b}(E_{0}\times E_{1})$ under the functor $\Phi$ with the equivariant Kuznetsov component $\mathcal{K}u_{M}^{G}$ and obtain the desired equivalence. However, for practical reasons, it is more convenient to use a functor in the other direction, i.e.\ let
$$\Psi\coloneqq [p_{*}\circ \textbf{L}q^{*}_{G}]^{G}:D^{b}_{G}(M)\longrightarrow D^{b}(Y),$$
where $p_{*}^{G}$ and $[\cdot]^{G}$ denote the equivariant pushforward and $G$-invarianrt functor.
Since $\Psi$ shares the same Fourier-Mukai kernel as $\Phi$, it is also an equivalence. 

After an explicit computation (see Proposition \ref{Theorem about a}, \ref{Theorem about b} and \ref{theorem about c}), we get a semiorthogonal decomposition of $D^{b}(Y)$ in terms of $\Psi(D^{b}_{G})$:
\begin{proposition}  
\begin{align*}
 D^{b}(Y)=  \langle &\Psi(\mathcal{K}u_{G}(M)), \CO_{Y},\CO_{Y}(1,-2,-1),\CO_{Y}(1,-1,-2),  
 \CO_{Y}(0,1,0),\\
&\CO_{Y}(1,-1,-1),\CO_{Y}(0,0,1),\CO_{Y}(0,2,0),\CO_{Y}(0,0,2),  \CO_{Y}(0,1,1)\rangle.       
\end{align*}      
\end{proposition}
Then by using mutation functors, we can finally identify the equivariant Kuznestsov component $\mathcal{K}u_{M}^{G}$ with $D^{b}(E_{0}\times E_{1})$.

\subsection*{Open questions}
 In \cite{laza2022automorphisms}, Laza and Zheng gave a complete classificaton of symplectic automorphism groups of cubic fourfolds and the case that we considered is one of them. It is natural to ask
 \begin{question}
  Can the derived Mckay correspondence be applied to other cases listed in \cite[Theorem 1.2 and 1.8]{laza2022automorphisms}?
 \end{question}
In general, if a cubic fourfold $M$ admits a finite symplectic automorphism group action of $G$, then the corresponding equivariant Kuznetsov component $\mathcal{K}u_{M}^{G}$ is known to be a $2$-Calabi-Yau category (see \cite[Section 6.3 \& 6.4]{beckmann2020equivariant} and \cite[Proposition 4.3]{beckmannOberdieck2020equivariant}). Since finding new examples of $2$-Calabi-Yau categories arising from geometry is an interesting task, we are wondering:
\begin{question}
Do equivariant Kuznetsov components $\mathcal{K}u_{M}^{G}$ provide us new examples of $2$-Calabi-Yau categories?
\end{question}
\begin{notations}
We always work over $\BC$. The bounded derived category of coherent sheaves on a smooth projective variety $X$ is denoted by $D^{b}(X)$.    
\end{notations}
\begin{ack}
This work is part of my master's thesis.
First and foremost, I thank my master's thesis advisors Georg Oberdieck and Dominique Mattei for answering my questions and helpful comments on earlier versions of this paper. This paper benefits from fruitful discussions with Jiexiang Huang, Adam Dauser and Tianyi Feng.
\end{ack}

{\Large{\part{{Preliminaries}}}}
\section{Semiorthogonal decomposition}
We recall some well-known facts about semiorthogonal decomposition and mutation functors closely following the treatment in \cite{2016Perry}. The reader can consult \cite{Bondal_1990}, \cite{Bondal_1990Asso} and \cite{huybrechts2006fourier} for more details.
\begin{definition}\label{definition 1}
Let $\Delta$ be a triangulated category. A \textit{semiorthogonal decomposition}
\begin{center}
    $\Delta=\,\langle \CA_{1},...,\CA_{n}\rangle$
\end{center}
is a sequence of full triangulated subcategories $\CA_{1},...,\CA_{n}$ of $\Delta$, which are called the \textit{components} of the decomposition, such that:
\begin{enumerate}
    \item For all $i\geq j$, $\CF_{i}\in \CA_{i}$, $\text{Hom}(\CF_{i},\CF_{j})=0$;
    \item For any $\CF\in\Delta$, there is a sequence of morphisms
    \begin{center}
        $0=\CF_{n}\rightarrow\CF_{n-1}\rightarrow \cdot \rightarrow \CF_{1}\rightarrow \CF_{0}=\CF$
    \end{center}
    such that Cone$(\CF_{i}\rightarrow\CF_{i-1})\in \CA_{i}$.
\end{enumerate}
Moreover, if $\text{Hom}(\CF_{i},\CF_{j})=\text{Hom}(\CF_{j},\CF_{i})=0$ for all $i,j,\CF_{i}\in \CA_{i}$ and $\CF_{j}\in\CA_{j}$, then we call $\langle \CA_{1},...,\CA_{n}\rangle$ a \textit{completely orthogonal decomposition} of $\Delta$.
\end{definition}\label{definition 2}
\begin{definition}\cite[page 3]{2016Perry}
A full triangulated subcategory $\alpha:\CA\hookrightarrow \Delta$ is called \textit{right admissible} if the inclusion functor $\alpha$ has a right adjoint $\alpha^{!}:\Delta\rightarrow\CA$, \textit{left admissible} if $\alpha$ has a left adjoint $\alpha^{*}\colon\Delta\rightarrow\CA$, and \textit{admissible} if it is both right and left admissible.
\end{definition}
If $\alpha:\CA\hookrightarrow\Delta$ is right admissible, then there is a semiorthogonal decomposition \begin{center}
    $\Delta=\,\langle\CA^{\bot},\,\CA \rangle$,
\end{center}
and if $\CA$ is left admissible, then there is a semiorthogonal decomposition
\begin{center}
    $\Delta=\,\langle \CA, \,^{\bot}\CA\rangle$.
\end{center}
Here $\CA^{\bot}$ resp.\,$^{\bot}\CA$ denotes the \textit{right} resp.\,\textit{left orthogonal} categories to $\CA$ respectively, defined as the full subcategories of $\Delta$ given by 
\begin{center}
    $\CA^{\bot}=\,\{\,\CF\in\Delta \mid$ Hom$(\CG,\CF)=0$ for all $\CG\in \CA\,\}$,\\
    $^{\bot}\CA=\,\{\,\CF\in\Delta \mid$ Hom$(\CF,\CG)=0$ for all $\CG\in \CA\,\}$.
\end{center}
\subsection{Mutations}

Let $\alpha\colon\CA\hookrightarrow\Delta$ be admissible, then for any object $\CF\in\Delta$, the counit morphism $\alpha\alpha^{!}(\CF)\rightarrow \CF$ can be completed to a distinguished triangle
\begin{center}
    $\alpha\alpha^{!}(\CF)\rightarrow\CF\rightarrow \CL_{\CA}(\CF)$,
\end{center}
where $\CL_{\CA}(\CF)$ is defined as the cone of the counit morphism. Similarly, for the unit morphism $\CF\rightarrow \alpha\alpha^{*}(\CF)$, there is also a distinguished triangle 
\begin{center}
    $\CR_{\CA}(\CF)\rightarrow\CF\rightarrow\alpha\alpha^{*}(\CF)$
\end{center}
 Since these triangles are functorial (see \cite[Remark 2.2]{2016Perry}), we can define the functors:
\begin{center}
    $\CL_{\CA}:\Delta\rightarrow\Delta$\, and \, $\CR_{\CA}:\Delta\rightarrow\Delta$
\end{center}
called the \textit{left} and \textit{right mutation functors} of $\CA\hookrightarrow\Delta$.

The restrictions:
\begin{center}
    $\CL_{\CA}\mid_{^{\bot}\CA}:\,^{\bot}\CA\rightarrow\,\CA^{\bot}$\,\,and \,\,$\CR_{\CA}\mid_{^{\bot}\CA}:\,\CA^{\bot}\rightarrow\,^{\bot}\CA$
\end{center}
are mutually inverse equivalences (by\cite[Lemma 1.9]{Bondal_1990}). 
The following proposition explains that
the mutation functors $\CL_{\CA},\CR_{\CA}$ act on semiorthogonal decompositions, which will turn out to be exceptionally useful for our later purpose.
\begin{proposition}\cite[Lemma 1.9]{Bondal_1990}
Let $\Delta=\,\langle \CA_{1},...,\CA_{n}\rangle$ be a semiorthogonal decompositions with admissible components. Then for $1\leq i\leq n-1$
\begin{center}
    $\Delta=\,\langle \CA_{1},...,\CA_{i-1},\CL_{\CA_{i}}(\CA_{i+1}),\CA_{i},\CA_{i+2},...,\CA_{n}\rangle$
\end{center}
 is a semiorthogonal decomposition, and for $2\leq i-2\leq n$
\begin{center}
    $\Delta=\,\langle \CA_{1},...,\CA_{i-2},\CA_{i},\CR_{\CA_{i}}(\CA_{i-1}),\CA_{i+1},...,\CA_{n}\rangle$
\end{center}
is a semiorthogonal decomposition.
\end{proposition}
We will recall several useful lemmata about mutation functors.
\begin{lemma}\cite[Lemma 2.4]{2016Perry}
Let $\Delta=\langle \CA_{1},...,\CA_{n}\rangle$ be a semiorthogonal decomposition with admissible components. Assume for some $i$ the components $\CA_{i}$ and $\CA_{i+1}$ are completely orthogonal, i.e. Hom $(\CF,\,\CG)=$ Hom $(\CG,\,\CF)=\,0$ for all $\CF\in\CA_{i}$, $\CG\in\CA_{i+1}$. Then $\CL_{\CA_{i}}(\CG)=\CG$ for any $\CG\in\CA_{i+1}$, and $\CR_{\CA_{i+1}}(\CF)=\CF$ for any $\CF\in \CA_{i}$. In particular,
\begin{center}
    $\Delta=\langle \CA_{1},...,\CA_{i-1},\CA_{i+1},\CA_{i},\CA_{i+2},...,\CA_{n}\rangle$
\end{center}
is a semiorthogonal decomposition.
\end{lemma}
\begin{proof}
For any $\CF_{i+1}\in\CA_{i+1}$, computing $\CL_{\CA_{i}}(\CF_{i+1})$ amounts to constructing a distinguished triangle
\begin{center}
    $\CF_{i}\rightarrow \CF_{i+1}\rightarrow \CG$ 
\end{center}
with $\CF_{i}\in\CA_{i}$ and $\CG\in\CA^{\bot}_{i}$, in which case $\CL_{\CA_{i}}(\CF_{i+1})=\CG$. Since $\text{Hom}(\CF_{i},\CF_{i+1})$ $=0$, $\CG \cong\CF_{i+1}$. The same argument applies to the corresponding statement on right mutation functors.
\end{proof}
\begin{lemma}\cite[Lemma 2.7]{Perry16alexander}
Let $\CA_{1},...,\CA_{n}$ is a semiorthogonal sequence of admissible subcategories of $\Delta$, then $\langle \CA_{1},...,\CA_{n}\rangle\hookrightarrow\Delta$ is also admissible and
    \begin{equation}
        \CL_{\langle \CA_{1},...,\CA_{n}\rangle}\,\cong\,\CL_{\CA_{1}}\circ\CL_{\CA_{2}}\circ\dots\circ\CL_{\CA_{n}};
    \end{equation}
    \begin{equation}
       \CR_{\langle \CA_{1},...,\CA_{n}\rangle}\,\cong\,\CR_{\CA_{n}}\circ\CR_{\CA_{n-1}}\circ\dots\circ\CR_{\CA_{1}}.
    \end{equation}
\end{lemma}
\hfill

We are mainly interested in the case when $\Delta$ is isomorphic to $D^{b}(X)$, where $X$ is a smooth projective variety and $D^{b}(X)$ is the derived category of coherent sheaves on $X$, which comes equipped with the Serre functor .

Let us first recall its definition:
\begin{definition}
Let $\Delta$ be a $\mathbb{C}$-linear triangulated category with finite dimensional Hom's. A \textit{Serre functor} for $\Delta$ is an equivalence $S:\Delta\longrightarrow\Delta$ together with a collection of binatural isomorphisms
\begin{center}
    $\eta_{A,B}:\,$Hom($\CF$, $\CG$)$\xrightarrow[]{\cong}$ Hom($\CG$, $S(\CF)$)$^{\vee}$ for $\CF$, $\CG$ $\in \Delta$.
\end{center}
\end{definition}
\begin{remark}
If $\Delta\,=D^{b}(X)$, where $X$ is a smooth projective $n$-dimensional variety, then by Serre duality, $D^{b}(X)$ admits a Serre functor given as 
\begin{center}
    $S(\CF)=\,\CF\otimes \omega_{X}[n]$
\end{center}
for $\CF\in D^{b}(X)$. Here $\omega_{X}$ is the canonical line bundle of $X$.
\end{remark}
\begin{proposition}\label{Serre functor}
Let $\Delta$ be a $\mathbb{C}$-linear triangulated category with finite dimensional Hom, which admits a Serre functor.
\begin{enumerate}
    \item If $\alpha:\CA\hookrightarrow\Delta$ is an admissible subcategory, then
    \begin{center}
        $S(^{\bot}\CA)$=$\CA^{\bot}$\,\,\,\,\,\, $S^{-1}$($\CA^{\bot}$)=$^{\bot}\CA$.
    \end{center}
    \item If $\Delta$ admits a semiorthogonal decomposition $\Delta=\,\langle \CA_{1},...,\CA_{n}\rangle$ with admissible components, then
    \begin{center}
    $\CL_{\langle \CA_{1},...,\CA_{n-1}\rangle}(\CA_{n})$= ($S(\CA_{n}$), $\CA_{1},...,\CA_{n-1}$)  
    \end{center}
    \begin{center}
    $\CR_{\langle \CA_{2},...,\CA_{n}\rangle}(\CA_{1})$= $(\CA_{2},...,\CA_{n},\,S^{-1}(\CA_{1})$)
    \end{center}
\end{enumerate}
\end{proposition}
\begin{proof}
See \cite[Proposition 3.6 and 3.7]{Bondal_1990}.
\end{proof}

\section{Equivariant derived categories}
In this section, we recall some facts about equivariant derived categories. In addition, we always work on $\BC$-linear categories and all functors are also $\BC$-linear. Our main references are \cite{beckmann2020equivariant}, \cite{beckmannOberdieck2020equivariant}, \cite{2016Perry},\cite{0f6d937083584c2ba1ca157da0f5ba0a} and \cite{elagin2015equivariant}, where the discussions are more detailed.
\subsection{Categorical actions}
 For the reader's convenience, we recall the definition of \textit{categorical actions} and \textit{equivariant categories} following \cite{beckmann2020equivariant}.
\begin{definition}\cite[Definition 2.1]{beckmann2020equivariant}
Let $G$ be a finite group and $\CD$ be a category.
A \textit{categorical action} $(\rho,\theta)$ of $G$ on $\CD$ consists of 
\begin{enumerate}
    \item for every $g\in G$, an autoequivalence $\rho_{g}\colon\CD \rightarrow \CD$;
    \item for every $g,h\in G$, an isomorphism of functors $\theta_{g,h}\colon\rho_{g}\circ \rho_{h} \rightarrow \rho_{gh}$ such that the following diagram 

\begin{equation}
\begin{tikzcd}
\rho_{g}\rho_{h}\rho_{k} \arrow[rr, "{\rho_{g}\circ\theta_{h,k}}"] \arrow[dd, "{\theta_{g,h}\circ\rho_{k}}"] &  & \rho_{g}\rho_{hk} \arrow[dd, "{\theta_{g,hk}}"] \\
                                                                                                   &  &                                                 \\
\rho_{gh}\rho_{k} \arrow[rr, "{\theta_{gh,k}}"]                                                    &  & \rho_{ghk}                                     
\end{tikzcd}  
\end{equation}
commutes for all $g,h,k\in G$. 
\end{enumerate}

We say a categorical action $(\rho,\theta)$ of $G$ on $\CD$ is \textit{trivial}, if for each $g\in G$ there exists a natural isomorphism $\tau_{g}:id \rightarrow \rho_{g}$, such that 
\begin{center}
    $\theta_{g,h}^{-1}\circ\tau_{gh}=\rho_{h}(\tau_{g})\circ\tau_{h}$
\end{center}
for all $g,h\in G$.
\end{definition}

\begin{definition}\cite[Definition 3.1]{beckmann2020equivariant}\label{definition of equivariant category}
Let $(\rho,\theta)$ be a categorical action of a finite group $G$ on an additive $\BC$-linear category $\CD$. The \textit{equivariant category} $\CD_{G}$ is defined as follows:
\begin{enumerate}
    \item Objects of $\CD_{G}$ are pairs $(E,\phi)$, where $E$ is an object in $\CD$ and the linearisation $\phi\coloneqq(\phi_{g})_{g\in G},\ \phi_{g}\colon E\mapsto \rho_{g}(E)$ is a family of isomorphisms such that the following diagram 
\begin{center}
    \begin{equation}
\begin{tikzcd}
E \arrow[rr, "\rho_{g}"] \arrow[rrrrrr, "\phi_{gh}", bend right] &  & \rho_{g}(E) \arrow[rr, "\rho_{g}(\phi_{h})"] &  & \rho_{g}\rho_{h}(E) \arrow[rr, "{\theta_{g,h}}(E)"] &  & \rho_{gh}(E)
\end{tikzcd}
    \end{equation}
\end{center}
commutes.

\item A morphism from $(E,\phi)$ to $(E',\phi')$ is a morphism $f\colon\,E\rightarrow E'$ in $\CD$ which commutes with linearizations, i.e.\ the following diagram
    \begin{equation}
\begin{tikzcd}
E \arrow[rr, "f"] \arrow[dd, "\phi_{g}"] &  & E^{\prime} \arrow[dd, "\phi_{g}^{\prime}"] \\
                                         &  &                            \\
\rho_{g}E \arrow[rr, "\rho_{g}(f)"]               &  & \rho_{g}E^{\prime}                      
\end{tikzcd}
    \end{equation}

commutes for every $g\in G$.
\end{enumerate}
\indent Note that for any objects $(E,\phi)$ and $(E',\phi')$ in $\CD_{G}$, there is an induced action of $G$ on Hom$_{\CD}(E,E')$ via
\begin{center}
 $g_{.}f=(\phi'_{g})^{-1}\circ \rho_{g}(f)\circ \phi_{g}$   
\end{center}
for every $g\in G$. Thus we have
\begin{center}
Hom$_{\CD_{G}}((E,\phi),(E',\phi'))$= Hom$_{\CD}(E,E')^{G}$. 
\end{center}

\end{definition}

\begin{remark}\cite[page 8]{beckmannOberdieck2020equivariant}\label{Remark about restrication functor}
The equivariant category $D_{G}$ is naturally equipped with a forgetful functor
\begin{center}
    $\text{Res}\colon\,\CD_{G} \rightarrow \CD,\,\,(E,\phi)\mapsto E$
\end{center}
and a linearization functor 
\begin{center}
    $\text{Ind}\colon\,\CD\rightarrow \CD_{G},\,\, E\mapsto(\bigoplus_{g\in G}\rho_{g}E,\phi)$.
\end{center}
Here the linearization $\phi$ is given by considering $\theta_{h,h^{-1}g}^{-1}\colon\,\rho_{g}E\mapsto \rho_{h}\rho_{h^{-1}g}E$ and then taking the direct sum over all $g$, \textit{i.e.}\ for each $h\in G$
\begin{center}
  $\phi_{h}\,=\bigoplus_{g}\theta^{-1}_{h,h^{-1}g}:\,\bigoplus_{g}\rho_{g}E\mapsto \rho_{h}(\bigoplus_{g}\rho_{h^{-1}g}E)=\,\rho_{h}(\bigoplus_{g}\rho_{g}E).$
\end{center}
By \cite[Lemma 3.8]{elagin2015equivariant}, the linearization functor $\text{Ind}$ is left and right adjoint to the forgetful functor $\text{Res}$. We write $E$ for $(E,\phi)$ if the linearization is of this form. 
\end{remark}

\subsection{Triangulated equivariant categories}
Take a finite group $G$ acting on a triangulated category $\Delta$ by exact autoequivalences. It is natural to ask: when is $\Delta_{G}$ a triangulated category?
The category $\Delta_{G}$ is triangulated only in certain situations. The following two circumstances are of particular interest to us:
\begin{enumerate}
    \item $\Delta=\,D^{b}(X)$ for a smooth projective variety $X$ and $G$ acts via automorphisms of $X$;
    \item $\Delta$ is a semiorthogonal component of $D^{b}(X)$ and $G$ acts via automorphisms of $X$ that preserve $\Delta$. 
\end{enumerate}
\begin{theorem}\cite[Theorem 6.3]{2011Elagin}\cite[Proposition 3.10]{elagin2015equivariant}\label{Equivariant Derived category Theorem}
Let $X$ be a quasi-projective variety with an action of a finite group $G$. Let $D^{b}(X)=\,\langle
\CA_{1},...,\CA_{n}\rangle$ be a semiorthogonal decomposition preserved by $G$, i.e.\ each $\CA_{i}$ is preserved by the action of $G$. Then there is a semiorthogonal decomposition
\begin{equation}
    D^{b}_{G}(X)=\,\langle \CA_{1,G},...,\CA_{n,G}\rangle
\end{equation}
of the equivariant category $D^{b}_{G}(X)$, where $\CA_{i,G}$ is the equivariant category of $\CA_{i}$.
\end{theorem}
For a semiorthogonal component $\CA$ of $D^{b}(X)$ preserved by $G$, the following proposition gives us a completely orthogonal decomposition of $\CA$ as long as $G$ induces a trivial action on $\CA$.
\begin{proposition}\cite{2016Perry}\label{KP16}
Let $\CA$ be a triangulated category with a trivial action of a finite group $G$. If the equivariant category $\CA_{G}$ is also triangulated, then there is a completely orthogonal decomposition
    \begin{equation}
        \CA_{G}=\langle \CA_{G}\otimes V_{0},...,\CA_{G}\otimes V_{n}\rangle,
    \end{equation}
where $V_{0},...,V_{n}$ are all irreducible representations of the finite group $G$.
\end{proposition}
\begin{example}
Take a point $*$ endowed with a trivial finite group action of $G$, then the equivariant derived category $D^{b}_{G}(*)$ is isomorphic to the derived category of representation $D^{b}(\text{Rep}_{\BC}(G))$. So it has a complete decomposition
\begin{equation*}
    \langle V_{0},V_{1},...,V_{n} \rangle,
\end{equation*}
where $V_{0},...,V_{n}$ are all the irreducible representation of $G$.
\end{example}
\begin{remark}\label{remark 3.6}
In particular, for a finite abelian group $G$, its irreducible representations one-to-one corresponds to characters. We also use $\CA_{G}\otimes \chi_{i}$ to denote $\CA_{G}\otimes V_{i}$, where $\chi_{i}$ is the character corresponding to the irreducible representation $V_{i}$. 
\end{remark}
\subsection{Equivariant sheaves and derived categories} In this section, we collect some facts from \cite[Section 2.2]{article} about the equivariant derived categories and functors that we will need later. 

Let a finite group $G$ act on a smooth projective variety $M$. We use $\text{Coh}_{G}(M)$ and $D^{b}_{G}(M)$ to denote the abelian category of equivariant coherent sheaves and the equivariant derived category. By  \cite[Theorem 9.6]{elagin2015equivariant} we have $D^{b}_{G}(M)\cong D^{b}(\text{Coh}_{G}(M))$.

For a non-trivial character $\chi$ of $G$, we get an autoequivalence 
\begin{center}
  $-\otimes \chi\colon \text{Coh}_{G}$ $(M)\rightarrow \text{Coh}_{G}(M)$.  
\end{center}

Let $G$ act on another smooth projective variety $N$ and $f\colon M\rightarrow N$ be a $G$-equivariant morphism. For our purposes, we can assume $f$ is projective. Then we have a natural pullback functor $f^{*}_{G}\colon \text{Coh}_{G}(N)\rightarrow \text{Coh}_{G}(M)$ and a push-forward functor $f_{*}^{G}\colon \text{Coh}_{G}(M)\rightarrow \text{Coh}_{G}(N)$. We also have the following two isomorphisms of functors, which play important roles later.
\begin{proposition}\label{technical and important proposition}
Let $M$ and $N$ be smooth projective varieties admitting actions of a finite group $G$. If $f\colon M\rightarrow N$ is a $G$-equivariant morphism and $\chi$ a character of $G$, then we have isomorphisms
\begin{center}
$f_{*}^{G}(-\otimes \chi)\cong f_{*}^{G}(-)\otimes \chi$, $f^{*}_{G}(-\otimes \chi)\cong f^{*}_{G}(-)\otimes \chi$.    
\end{center}
\end{proposition}
Furthermore, all functors mentioned above induce corresponding derived functors on the level of derived categories, we use $\textbf{L}f^{*}_{G}$ resp.\,$\textbf{R}f_{*}^{G}$ to denote the derived pull-back resp.\,derived push-forward $G$-equivariant functors. For the functor $-\otimes \chi$, we do not have to change notation for its derived counterpart, since it is exact.

For any two objects $\CF,\CG\in D^{b}_{G}(M)$, we denote the graded Hom-space by 
\begin{center}
$\text{Hom}_{G}^{*}(\CF,\CG)\coloneqq \bigoplus\limits_{i\in \BZ} \text{Ext}_{G}^{i}(\CF,\CG)$
and $\text{Ext}_{G}^{i}(\CF,\CG)\coloneqq \text{Hom}_{D^{b}_{G}(M)}(\CF,\CG[i])$.   
\end{center}
In addition, we write $\CF\coloneqq \text{Res}\,\CF$ for $\CF\in D^{b}_{G}(M)$ and $\text{Hom}^{*}(\CF,\CG)$ for $ \text{Hom}^{*}_{D^{b}(M)}$ $(\text{Res }\CF,\text{Res }\CG)$. By Remark \ref{Remark about restrication functor}, we know that
\begin{center}\label{Important isomorphism}
    $\text{Hom}_{G}^{*}(\CF,\CG)\cong \text{Hom}^{*}(\CF,\CG)^{G}$.
\end{center}

If $G$ acts trivially on $M$, then a $G$-equivariant sheaf $\CF$ is simply a sheaf of $G$-representations. Sectionwise, taking the $G$-invariants of $\CF(U)$ for every open subscheme $U\subset X$ yields a functor $(-)^{G}\colon \text{Coh}_{G}(M)\rightarrow \text{Coh}(M)$. Moreover, we also have the following isomorphisms between functors in this case. 
\begin{proposition}\label{important proposition 2}
If $f\colon M\rightarrow N$ is a morphism between varities and $G$ acts on $M$, $N$ and $f$ trivially, then we have the following isomophisms between functors
\begin{center}

 $(-)^{G}\circ f_{*}^{G}\cong f_{*}\circ(-)^{G}$ and $(-)^{G}\circ f^{*}_{G}\cong f^{*}\circ(-)^{G}$.   
\end{center}
\end{proposition}

\section{Derived Mckay correspondence and G-Hilbert schemes}
In \cite{0f6d937083584c2ba1ca157da0f5ba0a}, Bridgeland, King and Reid used derived categories to extend the classical McKay correspondence. In this section, we will review the main theorem in that paper. In the meantime, we also need to recall some basic facts about $G-$Hilbert schemes, which are key objects in both \cite{0f6d937083584c2ba1ca157da0f5ba0a} and this paper.

The classical McKay correspondence originated from an observation by McKay in \cite{mckay37graphs}. He found a bijection between non-trivial irreducible representations of a finite group $G\subset \text{SL}(2,\BC)$ and rational curves in the exceptional locus of the minimal resolution $Y\rightarrow \BC^{2}/G$ of the quotient singularities.

Later, in \cite{ASENS_1983_4_16_3_409_0}, Gonzalez-Springer and Verdier established an isomorphism between the Grothendieck group $K^{G}(\BC^{2})$ of $G$-equivariant coherent sheaves on $\BC^{2}$ and the Grothendieck group $K(Y)$ of the minimal resolution $Y$ of $\BC^{2}/G$. Since bounded derived categories can be thought as  \textit{``categorifications"} of Grothendieck groups, it is natural to expect that the isomorphism
\begin{equation*}
    K^{G}(\BC^{2})\cong K(Y)
\end{equation*}
can be lifted to an equivalence of derived categories. This has already been proved by Kapranov and Vasserot in \cite{kapranov1998kleinian}.
\begin{theorem}\cite[Section 1.4]{kapranov1998kleinian}
Let $M$ be a surface equipped with a holomorphic symplectic form $\omega$ and suppose that there is a $G$ finite action on $M$ preserving $\omega$. Then 
\begin{equation*}
    D^{b}(Y)\cong D^{b}_{G}(M),
\end{equation*}
where $Y\rightarrow M/G$ is the minimal resolution of $M /G$ and $D^{b}_{G}(M)$ is the derived category of $G$-equivariant coherent sheaves on $M$.
\end{theorem}

It is natural to suspect an extension of the derived Mckay correspondence to higher dimensional cases. However, a minimal resolution need not exist, even if $M$ is of dimension three. A natural replacement would be a crepant resolution, which is always minimal in dimension two. In \cite{10.3792/pjaa.72.135} and \cite{ito1999hilbert}, Ito and Nakumura introduced the $G$-Hilbert scheme $G\text{-Hilb}_{\BC}(M)$ as a candidate for a crepant resolution of $M/G$ and proved that if $\textup{dim}\,M =2$, then $G\text{-Hilb}_{\BC}(M)\rightarrow M/G$ is a crepant resolution.

Let us recall the definition of $G$-Hilbert schemes.
\begin{definition}
Let $S$ be a scheme and $M$ a $S$-scheme,
The \textit{$G$-Hilbert functor} of $M$ over $S$ $$\underline{G\text{-Hilb}}_{S}(M): (S-\text{schemes})^{op}\longrightarrow (\text{sets})$$ is defined by
\begin{equation*}
  \underline{G\text{-Hilb}}_{S}(M)(T) = \left\{  \begin{matrix}   \text{Quotient}\,G\text{-sheaves }[0\longrightarrow\CI\longrightarrow\CO_{M_{T}}\longrightarrow\CO_{Z}\longrightarrow 0]\\
   \text{ on }X_ T,\text{where } Z \text{ is finite flat over }T \text{, for every} \,t\in T\colon \\
   H^{0}(Z_{t},\CO_{Z_{t}})\,\text{is isomorphic}
\text{ to the regular representation}\\ 
 \text{ of }G

\end{matrix} \right\}
\end{equation*}
for an $S\text{-scheme}\,\,T$. If we take $T=S=\,\text{Spec}(\mathbb{C})$, every element in $G$-Hilb$_{\BC}(M)(\text{Spec}(\mathbb{C}))$ represents a $G$-cluster. Here a \textit{cluster} $Z\subset M$ is a zero-dimensional subscheme and a $G$-\textit{cluster} is a $G$-invariant cluster whose set of global sections $H^{0}(\CO_{Z},Z)$ is isomorphic to the regular representation $\mathbb{C}[G]$.
\end{definition}
\begin{proposition}\cite[Proposition 4.13]{blume2007mckay}
If $M$ is (quasi-)projective over $S$, then the functor $\underline{G\text{-Hilb}}_{S}(M)$ is represented by a (quasi-)projective $S$-scheme, which we denote by $G\text{-Hilb}_{S}(M)$.
\end{proposition}
\begin{remark}
For a smooth projective variety $M$ over $\BC$ with finite group action of $G$,
there is a \textit{Hilbert-Chow} morphism $\tau\colon G\text{-Hilb}_{\BC}(M)\longrightarrow X$, where $X\coloneqq M/G$. On the closed points, $\tau$ sends a $G$-cluster to the orbit supporting it. Moreover $\tau$ is a projective morphism, is onto and is birational on one irreducible component of $G\text{-Hilb}_{\BC}(M)$.\footnote{The reader can consult \cite{blume2007mckay} for the explicit construction of $\tau$ and \cite[Section 4.3]{10.2307/2373541} or \cite[Section 7.1]{fantechi2005fundamental} for the properties of $\tau$}
\end{remark}

When $\dim(M) > 2$, Bridgeland, King and Reid gave a criterion for the morphism $Y\rightarrow M/G$ to be crepant, in \cite{0f6d937083584c2ba1ca157da0f5ba0a}, where $Y$ is an irreducible component containing all free $G$-orbits in $G\text{-Hilb}_{\BC}(M)$. Moreover, they have shown the equivalence between derived category $D^{b}(Y)$ of $Y$ and the equivariant derived category $D^{b}_{G}(M)$ of $M$.

To be precise, let us recall the main theorem in \cite{0f6d937083584c2ba1ca157da0f5ba0a}:

Let $M$ be a nonsingular quasi-projective variety of dimension $n$, $G\subset \text{Aut}(M)$ be a finite group of automorphisms of $M$, with the property that the stabiliser subgroup of any point $x\in M$ acts on the tangent space $T_{x}(M)$ as a subgroup of $\text{SL}(T_{x}(M))$. With this condition, the quotient variety $X\coloneqq M/G$ has only Gorenstein singularities (see \cite[Introduction]{0f6d937083584c2ba1ca157da0f5ba0a}). Denote by $Y$ the irreducible component of $G\text{-Hilb}_{\BC}(M)$ that contains the free orbits. We consider the universal closed subscheme $\CZ\subset Y\times M$ and write $p$ and $q$ for its projections to $Y$, $M$ respectively. Then there is a commutative diagram of schemes
\begin{center}
\begin{tikzcd}
                      & \CZ \arrow[ld, "p"'] \arrow[rd, "q"] &                     \\
Y \arrow[rd, "\tau"'] &                                      & M \arrow[ld, "\pi"] \\
                      & X                                    &                    
\end{tikzcd}
\end{center}
in which $q$ and $\tau$ are birational, $p$ and $\pi$ are finite, and $p$, moreover, flat since $\CZ$ is a universal family of $Y$. Moreover, equipping $X$ and $Y$ with trivial $G$-actions, all morphisms in the above diagram are equivariant.

Define the functor 
\begin{center}
    $\Phi:=\,\textbf{R}q_{*}\circ p^{*}:\,D^{b}(Y)\longrightarrow D^{b}_{G}(M)$,
\end{center}
where a sheaf $\CF$ on $Y$ is viewed as a $G$-sheaf, endowed with the trivial action. Note that, as $p$ is flat, the pullback functor $p^{*}$ is already exact, so we do not need to derive.
\begin{theorem}\cite[Theorem 1.1]{0f6d937083584c2ba1ca157da0f5ba0a}\label{BKR 1}
Suppose that the fiber product 
\begin{center}
    $Y\times_{X}Y\,=\,\left\{(y_{1},y_{2})\in Y\times Y\,|\,\tau(y_{1})=\,\tau(y_{2})\right\}\subset Y\times Y$
\end{center}
has dimension $\leq n+1$. Then $Y$ is a crepant resolution of $X$ and $\Phi$ is an equivalence of categories.
\end{theorem}
\begin{remark}
When $\text{dim }M\leq\,3$ the condition of the theorem is automatically satisfied, since the dimension of the expectional locus of $Y\rightarrow X$ is smaller than 2. 
\end{remark}
If $\text{dim }M\leq\,3$, then the dimension of $Y\times_{X} Y$ depends heavily on the geometry of the $G$-Hilbert scheme $Y$. However, in general, understanding the geometry of $Y$ is hard. In our paper, we rely on the following result to ease this problem.
\begin{proposition}\cite[Proposition 2.40]{blume2007mckay}\label{Propostion 4.7}
Let $G=\,\text{Spec}\,\mathbb{C}[x]/(x^{r}-1)$ be the group scheme of $r$th roots of unity, and $V$ be an $n$-dimensional representation of $G$ over $\mathbb{C}$.

We define $\pi:\,M:=\,\mathbb{A}_{\mathbb{C}}(V)\longrightarrow X:=\,\mathbb{A}_{\mathbb{C}}(V)/G$, where $\mathbb{A}_{\mathbb{C}}(V)$ denotes the affine plane $\mathbb{A}_{\BC}^{n}$. Denote origin point of $\mathbb{A}_{\mathbb{C}}(V)$ by $0$. Then there is an isomorphism
\begin{center}
 $G\text{-Hilb}_{\BC}(M)\cong(\text{Bl}_{0}X)/G$
\end{center}
over $\BC$.
\end{proposition}
{\Large{\part{{Equivariant Kuznetsov component}}}}
Let $M:= V(F)\subset \BP^{5}$ be a smooth cubic fourfold, where $F=F_{0}(x_{0},x_{1},x_{2})+F_{1}(x_{3},x_{4},x_{5})$ for $F_{0}$, $F_{1}$ cubic equations and $G:=\,\BZ/3\BZ$. We pick a generator $g$ of $G$, whose action on $M$ is given by
\begin{center}
    $g_{.}[x_{0}:x_{1}:x_{2}:x_{3}:x_{4}:x_{5}]=[x_{0}:x_{1}:x_{2}:e^{2\pi i/3}x_{3}:e^{2\pi i/3}x_{4}:e^{2\pi i/3}x_{5}]$.
\end{center}
The quotient space $M/G$ is denoted by $X$ and we write $Y$ for the $G$-Hilbert scheme $G\text{-Hilb}_{\BC}M$.

Since $M$ is smooth, we know by the Jacobian criterion that $V(F_{0})\subset \BP^{2}$ and $V(F_{1})\subset \BP^{2}$ are smooth elliptic curves. For simplicity, we denote $V(F_{0})$ by $E_{0}$ and $V(F_{1})$ by $E_{1}$. Our main goal is to prove Theorem \ref{main theorem} stating 
\begin{equation*}
    D^{b}(E_{0}\times E_{1})\cong \mathcal{K}u_{G}(M),
\end{equation*}
where the equivariant Kuznetsov component $\mathcal{K}u_{G}(M)$ is defined in Section \ref{Section about overview of Step 2}.

Our proof for the main theorem \ref{main theorem} proceeds in two steps:
Firstly, we analyze the geometry of the $G$-Hilbert scheme $Y$. It turns out that $Y\cong \text{Bl}_{E_{0}\times E_{1}}(\BP^{2}\times \BP^{2})$ and satisfies the condition for the derived Mckay correspondence of Theorem \ref{BKR 1}.
Using this, we establish the category equivalence between the derived category of the $G$-Hilbert scheme $D^{b}(Y)$ and the equivariant derived category $D^{b}_{G}(M)$. Finally, we compute the image of components of $D^{b}_{G}(M)$ in $D^{b}(Y)$ and then by using mutation functors, we identify $\mathcal{K}u_{G}(M)$ with $D^{b}(E_{0}\times E_{1})$.

\section{Step I: The geometry of the G-Hilbert scheme Y}
\subsection{The main goal and central idea}
We want to construct the $G$-Hilbert scheme $Y$ explicitly. Since there is no systematic method to do this, we need to first guess the explicit description of $Y$ and then prove it is correct.

Optimistically, the $G$-Hilbert scheme would yield a crepant resolution of $M/G$. Intuitively speaking, the fixed point locus $M^{G}$ is the only obstruction to $M/G$ being smooth. To get our resolution of singularities, it is natural
to blow up the fixed locus first and then take the quotient of the group action.

By direct calculation, the fixed locus $M^{G}$ in $M$ can be identified with the union of two elliptic curves $E_{0}\times 0 \,\bigcup \,0\times E_{1}$, where 
\begin{equation*}
E_{0}\times 0\coloneqq\left\{
[x_{0}:x_{1}:x_{2}:x_{3}:x_{4}:x_{5}]\in M\,\mid\begin{matrix} \, [x_{0}:x_{1}:x_{2}]\in V(F_{0})\subset \BP^{2} \,\\
\text{and}\,x_{3}=x_{4}=x_{5}=0\end{matrix} \right\}
\end{equation*}
\begin{equation*}
0\times E_{1}\coloneqq \left\{
[x_{0}:x_{1}:x_{2}:x_{3}:x_{4}:x_{5}]\in M\,\mid\begin{matrix} \, [x_{3}:x_{4}:x_{5}]\in V(F_{1})\subset \BP^{2} \,\\
\text{and}\,x_{0}=x_{1}=x_{2}=0\end{matrix} \right\}
\end{equation*}
Now, let us study the quotient space $\text{Bl}_{M^{G}}(M)/G$ of the blow up $\text{Bl}_{M^{G}}(M)$ of the fixed locus $M^{G}$ on $M$. 
\subsection{The geometry of $\text{Bl}_{M^{G}}(M)/G$}
The first thing, that we should check is that our intuition holds water and $\text{Bl}_{M^{G}}(M)/G\longrightarrow X$ is a resolution, i.e.\,$\text{Bl}_{M^{G}}(M)/G$ is smooth over Spec($\BC$). 
\begin{proposition}\label{Proposition about smoothness of Y}
$\text{Bl}_{M^{G}}(M)/G$ is smooth over Spec$(\BC)$.
\end{proposition}
\begin{proof}
The idea of the proof is to compute $\text{Bl}_{M^{G}}(M)/G$ locally and then use the Jacobian criterion to show the smoothness of $\text{Bl}_{M^{G}}(M)/G$.

We observe that $\text{Bl}_{M^{G}}(M)/G$ is a closed subvariety of $\BLP/G$, since $\BP^{2}\times 0 \,\bigcap M\,= E_{0}\times 0$ and $0\times\BP^{2} \,\bigcap M\,=0\times E_{1}$. To obtain the local description of $\text{Bl}_{M^{G}}(M)/G$, our strategy is as follows:
We first do the computation for $\BLP$, then we use the equations of $M$ in $\BP^{5}$ to acquire the local equation of $\text{Bl}_{M^{G}}(M)$.

\noindent Step 1: We compute $\BLP$. Since $\BP^{2}\times 0 \bigcap 0\times \BP^{2}=\emptyset$, we have $\BLP= \text{Bl}_{\BP^{2}\times 0}(\text{Bl}_{0\times \BP^{2}}(\BP^{5}))$. And this allows us to decompose the computation into two easier steps.

Firstly, $\text{Bl}_{0\times \BP^{2}}(\BP^{5})$ is given by
\begin{equation*}
x_{3}y_{4}=x_{4}y_{3},\ \ x_{3}y_{5}=x_{5}y_{3},\ \  x_{4}y_{5}=x_{5}y_{4}    
\end{equation*}
in $\BP^{5}\times\BP^{2}$, where $[x_{0}:x_{1}:x_{2}:x_{3}:x_{4}:x_{5}]\times [y_{3}:y_{4}:y_{5}]$ are coordinates of $\BP^{5}\times \BP^{2}$. Then the further blow-up $\text{Bl}_{\BP^{2}\times 0}(\text{Bl}_{0\times \BP^{2}}(\BP^{5}))$ is cut out by the equations
\begin{align*}
&x_{0}y_{1}=x_{1}y_{0},\ \ x_{0}y_{2}=x_{2}y_{0},\ \ x_{1}y_{2}=x_{2}y_{1},\\ &x_{3}y_{4}=x_{4}y_{3},\ \  x_{3}y_{5}=x_{5}y_{3},\ \  x_{4}y_{5}=x_{5}y_{4}    
\end{align*}
in $\BP^{5}\times\BP^{2}\times \BP^{2}$ where $[x_{0}:x_{1}:x_{2}:x_{3}:x_{4}:x_{5}]\times [y_{0}: y_{1}: y_{2}]\times [y_{3}: y_{4}: y_{5}]$ are coordinates for $\BP^{5}\times \BP^{2}\times \BP^{2}$.

\noindent Step 2: Now we consider the quotient $\BLP/G$.
The blow-up $\BLP$ has a canonical affine cover
\begin{equation*}
 \{ U_{x_{i},y_{j},y_{k}}\mid 0\leq i\leq 5, \ 0\leq j \leq 2 \text{ and } 3\leq k \leq 5  \},
\end{equation*}
where $U_{x_{i},y_{j},y_{k}}$ is given by $x_{i}=y_{j}=y_{k}=1$ in $\BLP$.
Without loss of generality, we only do the computation for $U_{x_{0},y_{0},y_{3}}$. All the arguments below can be easily repeated and are essentially the same for the other affine subvarieties $U_{x_{i},y_{j},y_{k}}$.

The equations of $U_{x_{0},y_{0},y_{3}}$ in $\BA^{5}\times \BA^{2}\times \BA^{2}$ become
\begin{equation*}
 x_{1}=y_{1},\ x_{2}=y_{2},\ x_{4}=x_{3}y_{4},\ x_{5}=x_{3}y_{5}.  
\end{equation*}
Thus we have 
\begin{equation*}
 U_{x_{0},y_{0},y_{3}} \cong \text{Spec}(\BC[x_{1},x_{2},x_{3},y_{4},y_{5}])=\BA_{\BC}^{5}   
\end{equation*}
and we can take $x_{1},\,x_{2},\,x_{3},\,y_{4},\,y_{5}$ as the coordinates of $U_{x_{0},y_{0},y_{3}}$. The generator $g\in G$ acts on the coordinates $y_{4}$ and $y_{5}$ trivially, since 
\begin{equation*}
 g.x_{4}=(g.x_{3})y_{4},\ g.x_{5}=(g.x_{3})y_{5},   
\end{equation*}
so the action of $g$ on $U_{x_{0},y_{0},y_{3}}$ is given by
\begin{equation*}
(x_{1}, x_{2}, x_{3}, y_{4}, y_{5})\mapsto(x_{1}, x_{2},  e^{\frac{2\pi i}{3}}x_{3}, y_{4}, y_{5}).   
\end{equation*}
Thus \begin{equation*}
  \Uxyy/G\cong \text{Spec}(\BC[x_{1},x_{2},x_{3},y_{4},y_{5}]^{G})=\text{Spec}(\BC[x_{1},x_{2},x_{3}^{3},y_{4},y_{5}])\cong \BA_{\BC}^{5},   
 \end{equation*}
where for the last isomorphism, we write $x_{3}^{\prime}$ for $x_{3}^{3}$.

\noindent Step 3: Finally we restrict the previous computation to $\text{Bl}_{M^{G}}(M)/G$. We use $\Uxxyy$ to denote $\text{Bl}_{M^{G}}(M)/G\bigcap (\Uxyy/G)$. Then the equation of $\Uxxyy$ in $\Uxyy/G\cong \BA_{\BC}^{5}$ is:
\begin{equation}\label{Equation 1}
    F_{0}(1,x_{1},x_{2})+F_{1}(x_{3},x_{3}y_{4},x_{3}y_{5})= F_{0}(1,x_{1},x_{2})+x_{3}^{\prime}F_{1}(1,y_{4},y_{5})=0.
\end{equation}
Since there is no closed point in $\Uxxyy$ such that
\begin{align*}
 & F_{0}(1,x_{1},x_{2})=0,\ \frac{\partial F_{0}}{\partial x_{1}}=0,\ \frac{\partial F_{0}}{\partial x_{2}}=0, \\
 & F_{1}(1,x_{4},x_{5})=0,\  \frac{\partial F_{1}}{\partial x_{4}}=0,\ \frac{\partial F_{1}}{\partial x_{5}}=0,
\end{align*}
by the Jacobian criterion, $\Uxxyy$ is smooth over $\text{Spec}(\BC)$.
\end{proof}
We have the following global description of $\text{Bl}_{M^{G}}(M)/G$.
\begin{theorem}\label{Theorem 1}
$\text{Bl}_{E_{0}\times E_{1}}(\BP^{2}\times\BP^{2})\cong (\text{Bl}_{M^{G}}(M))/G$.
\end{theorem}
\begin{proof}
Let $[\Tilde{x}_{0}:\Tilde{x}_{1}:\Tilde{x}_{2}]\times[\Tilde{x}_{3}:\Tilde{x}_{4}:\Tilde{x}_{5}]$ be coordinates of $\BP^{2}\times\BP^{2}$. Recall that the equations of $E_{0}$ and $E_{1}$ are $F_{0}(\Tilde{x}_{0},\Tilde{x}_{1},\Tilde{x}_{2})=0$ and $F_{1}(\Tilde{x}_{3},\Tilde{x}_{4},\Tilde{x}_{5})=0$ respectively. Then the blow-up $\text{Bl}_{E_{0}\times E_{1}}(\BP^{2}\times\BP^{2})$ is cut out by the equation 
\begin{equation*}
\Tilde{y}_{1}F_{0}(\Tilde{x}_{0},\Tilde{x}_{1},\Tilde{x}_{2})+\Tilde{y}_{0}F_{1}(\Tilde{x}_{3},\Tilde{x}_{4},\Tilde{x}_{5})=0   
\end{equation*}
in $\BP^{2}\times \BP^{2}\times \BP^{1}$, where $[\Tilde{y}_{0}:\Tilde{y}_{1}]$ are the coordinates of $\BP^{1}$. Let $W_{i,l,m}$ be the affine open variety of $\BP^{2}\times \BP^{2}\times \BP^{1}$ given by $\Tilde{x}_{i}=\Tilde{x}_{l}=\Tilde{y}_{m}=1$. Write $W_{i,l,m}^{\prime}$ for the restriction of $\text{Bl}_{E_{0}\times E_{1}}(\BP^{2}\times\BP^{2})$ to $W_{i,l,m}$. 

Without loss of generality, we take the affine piece $W_{0,3,1}^{\prime}$ with its defining equation in $\BA^{5}_{\BC}$ given by
\begin{equation*}
 F_{0}(1,\Tilde{x}_{1},\Tilde{x}_{2})+\Tilde{y}_{0}F_{1}(1,\Tilde{x}_{4},\Tilde{x}_{5})=0.  
\end{equation*}
Let $\Uxxyy$ be the affine open subvariety of $(\text{Bl}_{M^{G}}(M))/G$ that we used in the proof of Proposition \ref{Proposition about smoothness of Y}.
Then there is an isomorphism between $W_{0,3,1}^{\prime}$ and $\Uxxyy$, given by 
\begin{equation*}
 (\Tilde{x}_{1},\Tilde{x}_{2},\Tilde{x}_{4},\Tilde{x}_{5},\Tilde{y}_{0})\mapsto (x_{1},x_{2},y_{4},y_{5},x_{3}^{\prime}). 
\end{equation*}
Similarly, one can check a similiar isomorphism for the other 17 pairs of affine open subvarieties from $\text{Bl}_{E_{0}\times E_{1}}(\BP^{2}\times\BP^{2})$ and $(\text{Bl}_{M^{G}}(M))/G$. Moreover, we can verify that these given local isomorphisms can be glued together. This yields an isomorphism as claimed.

\end{proof}
\subsection{The identification of $G\text{-Hilb}_{\BC}(M)$ with $\text{Bl}_{M^{G}}(M)/G$}
Until now we have already shown that $\text{Bl}_{M^{G}}(M)/G$ is smooth over $\text{Spec}(\BC)$. Moreover, it has a simple global description: $\text{Bl}_{E_{0}\times E_{1}}(\BP^{2}\times\BP^{2})$. Now we want to identify $\text{Bl}_{M^{G}}(M)/G$ with the $G$-Hilbert scheme $G\text{-Hilb}_{\BC}(M)$. 

Before delving into the details, let us briefly introduce the strategy we will use here. Since the $G$-Hilbert functor is stable under base change, it is enough to understand the $G$-Hilbert scheme of $G\text{-Hilb}_{\BC}(\BP^{5})$. 

We will start with figuring out the geometry of  $G\text{-Hilb}_{\BC}(\BA^{5})$, the open subset of $G\text{-Hilb}_{\BC}\\(\BP^{5})$ obtained by taking $x_{0}=1$ in the coordinates $[x_{0}:x_{1}:x_{2}:x_{3}:x_{4}:x_{5}]$ of $\BP^{5}$. Then the generator $g\in G$ acts on $\BA^{5}$ by mapping $(x_{1},x_{2},x_{3},x_{4},x_{5})$ to $(x_{1},x_{2},e^{2\pi i/3}x_{3},e^{2\pi i/3}x_{4},\\e^{2\pi i/3}x_{5})$. If we only consider the last three coordinates $(x_{3},x_{4},x_{5})$, we can identitify $\BA^{3}$ with a closed subscheme of $\BA^{5}$ via 
\begin{equation*}
\BA^{3}\cong 0\times \BA^{3}\subset \BA^{5}
\end{equation*}
and $\BA^{3}$ is endowed with a natural $G$-action given by
\begin{center}
$g_{.}(x_{3},x_{4},x_{5})=(e^{2\pi i/3}x_{3},e^{2\pi i/3}x_{4},e^{2\pi i/3}x_{5})$.
\end{center} 
Since $G$ acts on the coordinates $x_{1}$ and $x_{2}$ trivially, it is natural to expect that one can move this part out from $G\text{-Hilb}_{\BC}(\BA^{5})$. Explicitly, we have the following isomorphism:
\begin{lemma}\label{Lemma about G-Hilbet scheme}
$G\text{-Hilb}_{\BC}(\BA^{5}) \cong \BA^{2}\times G\text{-Hilb}_{\BC}(\BA^{3})$.
\end{lemma}
\begin{proof}
This is a direct corollary of \cite[Proposition 1.4.4]{terouanne:tel-00006683} or \cite[Corollary 4.24]{blume2007mckay}, where the following general statement is proved: Let $X$ and $Y$ be schemes with a $G$-action over $S$. If $Y$ is endowed with the trivial $G$-action, then 
\begin{center}
$\underline{G\text{-Hilb}}_{S}(X)\times_{S}\underline{Y}\cong \underline{G\text{-Hilb}}_{S}(X\times_{S} Y)$.   
\end{center}
\end{proof}

\begin{proposition}\label{remark 2}
The $G$-Hilbert scheme $G\text{-Hilb}_{\BC}(\BA^{5})$ is an open subscheme of $G\text{-Hilb}_{\BC}(\BP^{5})$ and the $G$-Hilbert scheme $G\text{-Hilb}_{\BC}(\BP^{5})$ is isomorphic to $(\BLP)/G$.
\end{proposition}
\begin{proof}
According to \cite[Remark 4.19]{blume2007mckay}, if $X^{\prime}$ is a scheme endowed with a $G$-action over $S$ and $S^{\prime}$ is an $S$-scheme, then there is an isomorphism of $S^{\prime}$-functors
\begin{center}
    $(\underline{G\text{-Hilb}}_{S}(X^{\prime}))_{S^{\prime}}\cong \underline{G\text{-Hilb}}_{S^{\prime}}(X^{\prime}_{S^{\prime}})$.
\end{center}
In addition, by \cite[Remark 4.22 (2)]{blume2007mckay}, we have 
\begin{center}
$G\text{-Hilb}_{\BC}(\BP^{5})\cong G\text{-Hilb}_{\BP^{5}/G}(\BP^{5})$ and $G\text{-Hilb}_{\BC}(\BA^{5})\cong G\text{-Hilb}_{\BA^{5}/G}(\BA^{5})$.
\end{center}
Thus we have the following Cartesian diagram:
$$\begin{tikzcd}
G\text{-Hilb}_{\BC}(\BA^{5})\cong G\text{-Hilb}_{\BA^{5}/G}(\BA^{5}) \arrow[dd, "\tau^{\prime}"] \arrow[rr, "g^{\prime}", hook] &  & G\text{-Hilb}_{\BC}(\BP^{5})\cong G\text{-Hilb}_{\BP^{5}/G}(\BP^{5}) \arrow[dd, "\tau"] \\
                                                                  &  &                                                       \\
\BA^{5}/G \arrow[rr, "g", hook]                                      &  & \BP^{5}/G,                                            
\end{tikzcd}$$
where $g$ and $g^{\prime}$ are open immersions. By Lemma \ref{Lemma about G-Hilbet scheme} and Proposition \ref{Propostion 4.7}, we obtain 
\begin{equation*}
 G\text{-Hilb}_{\BC}(\BA^{5})\cong \BA^{2}\times (\text{Bl}_{0}(\BA^{3})/G)\cong \text{Bl}_{\BA^{2}\times 0}(\BA^{5})/G.  
\end{equation*}
By taking other $x_{i}=1$ for $i=1,2,3,4,5$, we can get a Zariski cover of $G\text{-Hilb}_{\BC}(\BP^{5})$. Moreover, the gluing of this cover is completely determined by the open immersions
$\BA^{5}\hookrightarrow \BP^{5}$, and by a direct check, we get $$G\text{-Hilb}_{\mathbb{C}}(\mathbb{P}^{5})\cong (\BLP)/G.$$
\end{proof}
\begin{theorem}\label{Theorem 2}
$G\text{-Hilb}_{\BC}(M)\cong (\text{Bl}_{M^{G}}(M))/G\cong \text{Bl}_{E_{0}\times E_{1}}(\BP^{2}\times \BP^{2})$.
\end{theorem}
\begin{proof}
We proceed as in Proposition \ref{remark 2} and we get the Cartesian diagram:
$$\begin{tikzcd}
G\text{-Hilb}_{\BC}(M)\cong G\text{-Hilb}_{M/G}(M) \arrow[dd, "\tau^{\prime}"] \arrow[rr, "f^{\prime}", hook] &  & G\text{-Hilb}_{\BC}(\BP^{5})\cong G\text{-Hilb}_{\BP^{5}/G}(\BP^{5}) \arrow[dd, "\tau"] \\
                                                                  &  &                                                       \\
M/G \arrow[rr, "f", hook]                                      &  & \BP^{5}/G,                                           
\end{tikzcd}$$
where $f$ and $f'$ are closed immersions. On the other hand, by Remark \ref{remark 2}, $G\text{-Hilb}_{\BC}(\BP^{5})\cong \BLP$, so the fiber product in the diagram above can also be computed as 
\begin{equation*}
 (\BLP)/G\times_{\BP^{5}/G} M/G\cong \text{Bl}_{M^{G}}(M)/G.
\end{equation*}
 Hence, we obtain the isomorphism 
 \begin{equation*}
 G\text{-Hilb}_{\BC}(M)\cong (\text{Bl}_{M^{G}}(M))/G\cong \text{Bl}_{E_{0}\times E_{1}}(\BP^{2}\times \BP^{2})  
 \end{equation*}
 that we desired.
\end{proof}
\begin{theorem}\label{Theorem 3}
Let $Y$ be the $G$-Hilbert scheme $G\text{-Hilb}_{\BC}(M)$ and $X$ be the quotient space $M/G$. Then $Y$ is a crepant resolution of $X$. Moreover, we have an isomorphism:
$D^{b}(Y)\cong D^{b}_{G}(M)$.
\end{theorem}
\begin{proof}
By Theorem \ref{Theorem 2}, $Y\coloneqq G\text{-Hilb}_{\BC}(M)$ is isomorphic to $\text{Bl}_{E_{0}\times E_{1}}(\BP^{2}
\times \BP^{2})$, which is irreducible. Since the dimension of the fixed locus $M^{G}$ is 1, we have $\text{dim}(Y\times_{X} Y)\leq 5$. We take the family $\CZ\subset Y\times_{\BC} M$ to be $\text{Bl}_{M^{G}}(M)$ (see Remark \ref{Remark about Z}) to justify that it is indeed a family) and, via Theorem \ref{BKR 1}, we know that $Y$ is a crepant resolution of $X$ and $D^{b}(X)\cong D^{b}_{G}(M)$.
\end{proof}
\begin{remark}\label{Remark about Z}
In the statement of Theorem \ref{BKR 1}, $\CZ$ is the universal family of $Y\times_{\BC} M$. However, throughout the proof in \cite{0f6d937083584c2ba1ca157da0f5ba0a}, the universality of $\CZ$ never plays a role. We only need $\CZ$ to be a family of $\underline{G\text{-Hilb}}_{\BC}(M)$ over $Y$. We verify this condition for $\CZ\coloneqq \text{Bl}_{M^{G}}(M)$:
We define the map $j\colon\CZ\rightarrow Y\times M$ by $j=p\times q$, where $p\colon \text{Bl}_{M^{G}}(M)\rightarrow Y\cong\text{Bl}_{M^{G}}(M)/G$ is the quotient map and $q\colon \text{Bl}_{M^{G}}(M)\rightarrow M$ is the blow-up morphism. By a direct check, we see that $j$ is a closed immersion.
\end{remark}
We conclude this section with an alternative proof of the fact that $Y$ is a crepant resolution of $X$. Although this is a direct corollary of Theorem \ref{BKR 1}, we can also show this result independently via direct computation.
We present it here for interested readers.

Before stating of the proposition, we need to recall the definition of canonical singularities:
\begin{definition}\cite[Definition 1.1]{reid1985young}\label{Definition 5.10}
Let $X$ be a normal, quasi-projective variety over an algebraic closed field $k$ of characteristic $0$. We say $X$ has \textit{canonical singularities}, if it satisfies the following two conditions:
\begin{enumerate}
    \item The canonical class $K_{X}$ is $\BQ$-Cartier, i.e. for some integer $r\geq 1$, the Weil divisor $rK_{X}$ is Cartier;
    \item For any resolution $f:Y \longrightarrow X$, 
    \begin{center}
        $rK_{Y}= f^{*}(rK_{X})+\sum a_{i}E_{i}$
    \end{center}
    for some $a_{i}\geq 0$, where $E_{i}$'s are exceptional divisors.
\end{enumerate}
\end{definition}
\begin{example}\cite[Proposition 1.7 and Remark 1.8]{reid1980canonical}\label{Lemma 5.11}
If a finite group quotient singularity is Gorenstein, then it is also canonical.
\end{example}

\begin{proposition}
$Y:= (\text{Bl}_{M^{G}}(M))/G \underset{}{\overset{\tau}{\longrightarrow}} X:= M/G$ is a crepant resolution of $X$.
\end{proposition}

\begin{proof}
Recall the commutative diagram from our previous discussion:
\begin{center}

\begin{tikzcd}
\CZ= \text{Bl}_{M^{G}}(M)\arrow[rr, "q"] \arrow[dd, "p"]      &  & M \arrow[dd, "\pi"] \\
                                         &  &                     \\
Y= \text{Bl}_{M^{G}}(M)/G\arrow[rr, "\tau"] &  & X:=M/G.             
\end{tikzcd}
\end{center}
By a direct check, the stabiliser subgroup of any point $x\in M$ acts on the tangent space $T_{x}M$ as a subgroup of $\text{SL}(T_{x}M)$, hence $X$ has only Gorenstein singularities (see \cite[Introduction]{0f6d937083584c2ba1ca157da0f5ba0a}). We write  $U:=\text{im}\,\pi(M \backslash M^{G})$ and $l_{*}\omega_{U}$ becomes a line bundle, where $l\colon U\hookrightarrow X$ is the natural open immersion, then the canonical line bundle $\omega_{X}= l_{*}\omega_{U}$. Since $G$ acts on $U$ freely, the morphism $\pi|_{\pi^{-1}(U)}$ is \'etale and we have $\pi^{*}\omega_{U}\cong w_{M}|_{\pi^{-1}(U)}$. Moreover, $M$ is smooth and the codimension of $M^{G}$ is 3, so we get $\pi^{*}K_{X}\cong K_{M}$. By the fact that $q$ is a birational morphism, we have
\begin{equation*}
 K_{\CZ}\cong q^{*}K_{M}+ 2E, 
\end{equation*}
where $E$ is the exceptional divisor. 

By Definition \ref{Definition 5.10} and Example \ref{Lemma 5.11}, we know that X has canonical singularities, so $K_{Y}=\tau^{*}K_{X}+\sum a_{i}E_{i}$, where $E_{i}$ are prime effective divisors of $Y$ and $a_{i}\geq 0$ for all $i$. The next step is to calculate $p^{*}K_{Y}$. Since $p$ is a finite morphism, by the adjunction formula, we know that 
\begin{equation*}
 K_{\CZ}=p^{*}K_{Y}+ \Tilde{E}, 
\end{equation*}
 where $\Tilde{E}$ is the ramification divisor of $p$. By taking $x_{0}=y_{0}=y_{3}=1$, the associated ring map of $p$ is
 \begin{align*}
 \frac{\BC[x_{1},x_{2},x_{3}^{3},y_{4},y_{5}]}{(F_{0}(1,x_{1},x_{2})+x_{3}^{3}F_{1}(1,y_{4},y_{5}))}\longrightarrow 
 \frac{\BC[x_{1},x_{2},x_{3},y_{4},y_{5}]}{( F_{0}(1,x_{1},x_{2})+x_{3}^{3}F_{1}(1,y_{4},y_{5}))}.
 \end{align*}
Here we use the coordinates appearing in the proof of Proposition \ref{Proposition about smoothness of Y}. Then the ramification divisor $\Tilde{E}$ is the support of $\Omega_{\CZ/Y}|_{x_{0}=y_{0}=y_{3}=1}$, which is $2V(x_{3})= 2(E_{0}\times 0)|_{x_{0}=1}\times \BA_{\BC}^{2}$. By similar computations, it is not hard to verify that the ramification divisor is $\Tilde{E}=2E$, where $E$ is the exceptional divisor of $q$. Moreover, since the ring map
\begin{equation*}
\BC[x_{1},x_{2},x_{3}^{3},y_{4},y_{5}]\longrightarrow \BC[x_{1},x_{2},x_{3},y_{4},y_{5}]    
\end{equation*}
is flat and flat morphism is stable under base change, we obtain $p$ is flat. 
Hence the pullback $p^{*}(E_{i})$ is still an effective divisor, which is determined by $p^{-1}(E_{i})$. Since $q^{*}(\pi^{*}K_{X})=p^{*}(\tau^{*}K_{X})$, we have
\begin{equation*}
 K_{\CZ}+2E= K_{\CZ}+2E-\sum_{i} a_{i}p^{-1}(E_{i}),   
\end{equation*}
hence $a_{i}$ has to be $0$ for all $i$. This proves $\tau$ is indeed a crepant resolution of $X$.
\end{proof}
\section{Step II: The equivariant Kuznetsov component of M}
The goal of this section is to identify the equivariant Kuznetsov component $\mathcal{K}u_{M}^{G}$ of $D^{b}_{G}(M)$, which we define in Section \ref{Section about overview of Step 2}, with $D^{b}(E_{0}\times E_{1})$.
\subsection{An overview of the strategy of Step II}\label{Section about overview of Step 2}
The purpose of this section is to review some necessary notions and explain the main ideas to achieving our goal.

The derived category $D^{b}(M)$ of the cubic fourfold $M$ admits a semiorthogonal decomposition, which was first introduced by Alexander Kuznetsov in \cite{kuznetsov2007homological}:
\begin{equation}
  D^{b}(M)= \langle \mathcal{K}u_{M},\CO_{M},\CO_{M}(1),  \CO_{M}(2)\rangle,  
\end{equation}
where the Kuznetsov component $Ku_{M}$ is by definition the orthogonal complement $\langle \CO_{M},\CO_{M}(1),$ $\CO_{M}(2) \rangle^{\bot}$.
Furthermore, by Theorem \ref{Equivariant Derived category Theorem}, the equivariant derived category $D^{b}_{G}(M)$ also admits a semiorthogonal decomposition:
\begin{equation}\label{equation 10}
  D^{b}_{G}(M)=\langle \mathcal{K}u_{M}^{G}, \langle\CO_{M}\rangle^{G},\langle\CO_{M}(1)\rangle^{G},\langle\CO_{M}(2)\rangle^{G}\rangle, 
\end{equation}
where the equivariant Kuznetsov component $\mathcal{K}u_{M}^{G}$ by definition the equivariant triangulated category of $\mathcal{K}u_{M}$.

On the other hand, from Theorem \ref{Theorem 2}, we know $Y\cong \text{Bl}_{E_{0}\times E_{1}}(\BP^{2}\times \BP^{2})$. By Orlov's blow-up formula \cite{Orlov_1993} and \cite[Proposition 11.16 and 11.18]{huybrechts2006fourier}, we have 
\begin{equation}\label{}
   D^{b}(Y)=\langle D_{-1}, D_{0}\rangle. 
\end{equation}
where $D_{-1}\cong D^{b}(E_{0}\times E_{1})$ and $D_{0}\cong D^{b}(\BP^{2}\times \BP^{2})$. (See Section \ref{Section 5.2} below for definitions of equivalences.)

Recall from Theorem \ref{BKR 1} that we have the following commutative diagram:
\begin{center}
\begin{tikzcd}
                      & \CZ \arrow[ld, "p"'] \arrow[rd, "q"] &                     \\
Y \arrow[rd, "\tau"'] &                                      & M \arrow[ld, "\pi"] \\
                      & X
                      &                    .
\end{tikzcd}   
\end{center}
Here $\CZ$ is $\text{Bl}_{M^{G}}(M)$, as explained in Remark \ref{Remark about Z}.
For the purpose of computation, we will use the functor in the opposite direction with the same Fourier-Muaki kernel of $\Phi$, namely
\begin{center}
$\Psi\coloneqq p_{*}\circ \textbf{L}q^{*}\colon D_{G}^{b}(M)\longrightarrow D^{b}(Y)$.    
\end{center}
Since the morphism $p$ is finite, in particular affine, $p_{*}$ is an exact functor. This justifies writing $p_{*}$ instead of $\textbf{R}p_{*}$.
By computing the image of the three components $\langle\CO_{M}\rangle^{G},\langle\CO_{M}(1)\rangle^{G}$ and $\langle\CO_{M}(2)\rangle^{G}$ of $D^{b}_{G}(M)$ under $\Psi$ and using mutation functors, we can finally find an explicit functor to identify $\mathcal{K}u_{M}^{G}$ with $D^{b}(E_{0}\times E_{1})$.
\subsection{$D^{b}(Y)$, $D_{G}^{b}(M)$ and the functor $\Psi$}\label{Section 5.2}
We first need a refined semiorthogonal decompositions of $D^{b}(Y)$ and $D^{b}_{G}(M)$.

Following \cite{Orlov_1993} and \cite[Proposition 11.16 and 11.18]{huybrechts2006fourier} for the further decomposition of $D^{b}(Y)$, we utilize the diagram
\begin{center}
\begin{tikzcd}
                  & E' \arrow[ld, "\pi'"'] \arrow[rd, "i'", hook] &             \\
E_{0}\times E_{1} &                                             & Y=\text{Bl}_{E_{0}\times E_{1}}(\BP^{2}\times \BP^{2})
\end{tikzcd}
\end{center}
where $E^{\prime}$ is the exceptional divisor of the blow-up $p^{\prime}:=Y\longrightarrow \BP^{2}\times\BP^{2}$. The morphism $i^{\prime}$ is the closed immersion from $E^{\prime}$ to $Y$ and $\pi^{\prime}$ is the natural projection. The subcategory $D_{-1}$ is given by $\phi_{-1}(D^{b}(E_{0}\times E_{1}))$, where
\begin{center}
$\phi_{-1}:= i_{*}\circ(\CO_{E'}(-E')\otimes (\cdot))\circ \pi^{*}\colon D^{b}(E_{0}\times E_{1})\longrightarrow D^{b}(Y)$. 
\end{center}
 Moreover, $D_{-1}$ is equivalent to $D^{b}(E_{0}\times E_{1})$ by $\phi_{-1}$. $D_{0}$ is given by the image of $(p^{\prime})^{*}D^{b}(\BP^{2}\times \BP^{2})$ and equivalent to $D^{b}(\BP^{2}\times \BP^{2})$ via $(p^{\prime})^{*}$. In particular, $\langle D_{-1},D_{0}\rangle$ is a semiorthogonal decomposition of $D^{b}(Y)$.

Consider the diagram
\begin{center}
\begin{tikzcd}
\BP^{2}\times \BP^{2} \arrow[dd, "p_{1}"] \arrow[rr, "p_{2}"] &  & \BP^{2} \\
                                                                &  &         \\
\BP^{2}                                                         &  &        
\end{tikzcd} 
\end{center}
where $p_{i}$ is the projection to the $i$-th factor. By \cite{Orlov_1993} and \cite[Proposition 8.28 and Corollary 8.36]{huybrechts2006fourier}, we get that
\begin{align*}
D_{0}\cong D^{b}(\BP^{2}\times \BP^{2})=
& \langle p_{2}^{*}D^{b}(\BP^{2})\otimes \CO_{\BP^{2}\times \BP^{2}},p_{2}^{*}D^{b}(\BP^{2})\otimes p_{1}^{*}\CO_{\BP^{2}}(1),\\
& p_{2}^{*}D^{b}(\BP^{2})\otimes p_{1}^{*}\CO_{\BP^{2}}(2)\rangle.   
\end{align*}
Moreover, by \cite{Beilinson1978CoherentSO}, we know that $D^{b}(\BP^{2})=\langle \CO_{\BP^{2}},\CO_{\BP^{2}}(1),\CO_{\BP^{2}}(2)\rangle$, so $D^{b}(\BP^{2}\times \BP^{2})$ has the decomposition:
\begin{equation}\label{Equation 11}
\begin{split}
D^{b}(\BP^{2}\times \BP^{2})=  \langle \CO_{\BP^{2}\times \BP^{2}}(0,0),\CO_{\BP^{2}\times \BP^{2}}(0,1),\CO_{\BP^{2}\times \BP^{2}}(0,2),\CO_{\BP^{2}\times \BP^{2}}(1,0),\\\CO_{\BP^{2}\times \BP^{2}}(1,1), \CO_{\BP^{2}\times \BP^{2}}(1,2),\CO_{\BP^{2}\times \BP^{2}}(2,0),\CO_{\BP^{2}\times \BP^{2}}(2, 1),\CO_{\BP^{2}\times \BP^{2}}(2,2)\rangle,  
\end{split}
\end{equation}
where $\CO_{\BP^{2}\times \BP^{2}}(i,j)$ denotes $p_{1}^{*}\CO_{\BP^{2}}(i)\otimes p_{2}^{*}\CO_{\BP^{2}}(j)$, and we write $\CO_{Y}(0,i,j)$ for $(p^{\prime})^{*}(p_{1}^{*}\CO_{\BP^{2}}(i)\otimes p_{2}^{*}\CO_{\BP^{2}}(j))$ for convenience.

Combining the decomposition of $D_{-1}$ and \ref{Equation 11}, we obtain a finer semiorthogonal decomposition of $D^{b}(Y)\colon$
\begin{equation}\label{equation 12}
    \begin{split}
       D^{b}(Y)= & \langle D_{-1},\CO_{Y}(0,0,0),\CO_{Y}(0,0,1),
       \CO_{Y}(0,0,2),
       \CO_{Y}(0,1,0),\\
      &\CO_{Y}(0,1,1),\CO_{Y}(0,1,2),\CO_{Y}(0,2,0),\CO_{Y}(0,2, 1),\CO_{Y}(0,2,2)\rangle.
    \end{split}
\end{equation}

At the same time, we know 
\begin{equation*}
D^{b}_{G}(M)=\langle \mathcal{K}u_{M}^{G},\langle\CO_{M}\rangle^{G},\langle\CO_{M}(1)\rangle^{G}\langle\CO_{M}(2)\rangle^{G}\rangle.   
\end{equation*}
 And by \ref{KP16}, we get
 \begin{center}
$\langle\CO_{M}(n)\rangle^{G}=\langle \CO_{M}(n)\otimes V_{0},\CO_{M}(n)\otimes V_{1},\CO_{M}(n)\otimes V_{2}$ $\rangle$,     
 \end{center}
where $V_{i}$'s range over irreducible representations of $G$. For $G=\BZ/3\BZ$, they are given by the characters $\chi_{0}$, $\chi_{1}$ and $\chi_{2}$.

By the discussion above, we have the following semiorthogonal decomposition of $D^{b}_{G}(M)$:
\begin{equation}\label{equation 13}
\begin{split}
    & D^{b}_{G}(M)=  \langle \mathcal{K}u_{M}^{G}, \CO_{M}\otimes \chi_{0}, \CO_{M}\otimes \chi_{1},
    \CO_{M}\otimes \chi_{2}, \CO_{M}(1)\otimes \chi_{0},\\
    & \CO_{M}(1)\otimes \chi_{1},
     \CO_{M}(1)\otimes \chi_{2}, \CO_{M}(2)\otimes \chi_{0},\CO_{M}(2)\otimes \chi_{1},\CO_{M}(2)\otimes \chi_{2} \rangle.
\end{split} 
\end{equation}
And $\CO_{M}(i)\otimes \chi_{0}$ is taken with the canonical $G$-structure such that 
\begin{equation*}
H^{0}(M,\CO_{M}(1))= \bigoplus\limits_{k=0}^{5} \BC x_{k}= \BC^{3}\chi_{0}\bigoplus \BC^{3}\chi_{1}.
\end{equation*}
We observe that besides subcategories $D^{b}(E_{0}\times E_{1})\subset D^{b}(Y)$ and $\mathcal{K}u_{M}^{G}\subset D^{b}_{G}(M)$, that we are after, there are exactly nine terms in the semiorthogonal decomposition \eqref{Equation 11} and \eqref{equation 13} both for $D^{b}(Y)$ and $D^{b}_{G}(M)$. Using the functor $\Phi$ of \ref{BKR 1} to identify these nine terms with each other on both sides seems to be the most natural idea. However, for the purpose of practical computation, it is better to use 
\begin{center}
$\Psi\coloneqq p_{*}\circ \textbf{L}q^{*}:D^{b}_{G}(M)\longrightarrow D^{b}(Y)$
\end{center}
instead. Here $\textbf{L}q^{*}\colon D^{b}_{G}(M)\longrightarrow D^{b}_{G}(\CZ)$ is the derived functor induced by the $G$-equivariant map $q\colon \CZ\rightarrow M$, and $p_{*}$ is the composition
\begin{center}
$D^{b}_{G}(\CZ)\underset{}{\overset{p_{*}^{G}}{\longrightarrow}} D^{b}_{G}(Y)\underset{}{\overset{[\cdot]^{G}}{\longrightarrow}} D^{b}(Y)$,    
\end{center}
where $p_{*}^{G}$ denotes the equivariant pushforward and $G$ acts trivially on $Y$.

The functor $\Psi$ shares the same Fourier-Mukai kernel with $\Phi$. Moreover, since $\Phi$ is an equivalence, so is $\Psi$. (See \cite[Proposition 2.9]{article}) 
\subsection{Computing $\Psi(\CO_{M}(i)\otimes \chi_{j})$}

For the reader's convenience, we recall the commutative diagram defined in \ref{BKR 1}:
\begin{center}
\begin{tikzcd}
                      & \CZ\coloneqq \text{Bl}_{M^{G}}(M) \arrow[ld, "p"'] \arrow[rd, "q"] &                     \\
Y \arrow[rd, "\tau"'] &                                      & M \arrow[ld, "\pi"] \\
                      & X
                      &                    .
\end{tikzcd}   
\end{center}
The pullback $\textbf{L}q^{*}$ of $\CO_{M}(i)\otimes \chi_{j}$ is still an equivariant line bundle and the equivariant pushforward $p_{*}^{G}(\textbf{L}q^{*}(\CO_{M}(i)\otimes \chi_{j}))$ is a rank $3$ equivariant vector bundle since $p$ is flat and finite of degree $3$. Therefore, after taking $G$-invariants, we have the following result
\begin{lemma}
$\Psi(\CO_{M}(i)\otimes \chi_{j})$ is a line bundle on $Y$.
\end{lemma}
\begin{proof}
$G$ acts trivially on $Y$, by \cite[Section 4.2]{0f6d937083584c2ba1ca157da0f5ba0a}, we know
\begin{equation*}
p_{*}^{G}(\textbf{L}q^{*}(\CO_{M}(i)\otimes \chi_{0}))\cong \bigoplus\limits_{j=0}^{3} p_{*}\textbf{L}q^{*}(\CO_{M}(i)\otimes \chi_{0})\otimes \chi_{j}.   
\end{equation*}
It follows from Proposition \ref{technical and important proposition} that
\begin{equation*}
 p_{*}\textbf{L}q^{*}(\CO_{M}(i)\otimes \chi_{0})\otimes \chi_{j}\cong  p_{*}\textbf{L}q^{*}(\CO_{M}(i)\otimes \chi_{j}).   
\end{equation*}
Hence,
\begin{equation*}
p_{*}^{G}(\textbf{L}q^{*}(\CO_{M}(i)\otimes \chi_{0}))\cong \bigoplus\limits_{j=0}^{3} p_{*}\textbf{L}q^{*}(\CO_{M}(i)\otimes \chi_{j})= \bigoplus\limits_{j=0}^{3} \Psi(\CO_{M}(i)\otimes \chi_{j}).
\end{equation*}
Since $\Psi(\CO_{M}(i)\otimes \chi_{j})$ is a direct summand of the vector bundle $p_{*}^{G}(\textbf{L}q^{*}(\CO_{M}(i)\otimes \chi_{0}))$, it is a locally free sheaf.
Moreover, $\Psi$ is an equivalence, so $\Psi(\CO_{M}(i)\otimes \chi_{j})$ is not $0$ and must be a line bundle.
\end{proof}
 By Theorem \ref{Theorem 2}, we know that $Y$ is isomorphic to $\text{Bl}_{E_{0}\times E_{1}}(\BP^{2}\times \BP^{2})$, which is smooth over $\BC$. So we know that this line bundle on $Y$ is uniquely determined by its corresponding Weil divisor. The Picard group of $\text{Bl}_{E_{0}\times E_{1}}(\BP^{2}\times \BP^{2})$ also has an explicit characterisation. 
 \begin{proposition}
 Let $H_{i}$ denote the divisor of the pullback $p_{i}^{*}(\CO_{\BP^{2}}(1))$ for $i=1,2$ and $E^{\prime}$ denote the exceptional divisor of the blow-up $\text{Bl}_{E_{0}\times E_{1}}(\BP^{2}\times \BP^{2})$. Then
 \begin{center}
 $\text{Pic}(Y)\cong  \BZ E^{\prime}\bigoplus \text{Pic}(\BP^{2}\times \BP^{2})\cong \BZ E^{\prime}\bigoplus \BZ H_{1}\bigoplus \BZ H_{2}$.    
 \end{center}
 \end{proposition}
\begin{proof}
By \cite[Chapter II Exercises 7.9 and 8.5]{hartshorne2013algebraic}.
\end{proof}

We take $\BP^{1}$ as a fiber of $E^{\prime}\rightarrow E_{0}\times E_{1}$ and a line in $\BP^{2}$. The image $\Psi(\CO_{M}(i)\times \chi_{j})|_{\BP^{1}}$ for different $\BP^{1}$'s can then be computed in order to get degree of the corresponding divisor, i.e. $E,H_{1}\text{ and } H_{2}$. In the end we can then recover all images $\Psi(\CO_{M}(i)\otimes \chi_{j})$ in terms of divisors.
\begin{definition}\label{definition about a,b,c}
We define three integers $a_{i,j}, b_{i,j}, c_{i.j}$ which satisfy $\Psi(\CO_{M}(i)\otimes \chi_{j}) = \CO_{M}(a_{i,j}E^{\prime}+b_{i,j}H_{1}+c_{i,j}H_{2})$. 
\end{definition}

According to the strategy, we split our computation into three parts.
\subsubsection{The coefficients $a_{i,j}$}\label{Subsubsection 6.3.1}
Recall from the proof of Theorem \ref{Theorem 2}, that
the blow-up $\text{Bl}_{E_{0}\times E_{1}}\\(\BP^{2}\times \BP^{2})$ is cut out by
\begin{equation}\label{Eqaution of Y'}
\Tilde{y}_{1}F_{0}(\Tilde{x}_{0},\Tilde{x}_{1},\Tilde{x}_{2})+\Tilde{y}_{0}F_{1}(\Tilde{x}_{3},\Tilde{x}_{4},\Tilde{x}_{5})=0   
\end{equation}
in $\BP^{2}\times \BP^{2}\times \BP^{1}$, where $[\Tilde{x}_{0}:\Tilde{x}_{1}:\Tilde{x}_{2}]\times[\Tilde{x}_{3}:\Tilde{x}_{4}:\Tilde{x}_{5}]\times [\Tilde{y}_{0}:\Tilde{y}_{1}]$ are coordinates of $\BP^{2}\times \BP^{2}\times \BP^{1}$. Take a point $x\in E_{0}\times E_{1}$. Then we have the following fiber product diagram:
\begin{center}
\begin{tikzcd}
\BP^{1}_{x} \arrow[rrr, hook] \arrow[dd] &  &  & \text{Bl}_{E_{0}\times E_{1}}(\BP^{2}\times \BP^{2}) \arrow[dd, "p^{\prime}"] \\
                                         &  &  &                                                                               \\
x \arrow[rrr, "j^{\prime}", hook]        &  &  & \BP^{2}\times \BP^{2},                                                        
\end{tikzcd}   
\end{center}
where $j^{\prime}$ is a closed immersion and $p^{\prime}$ denotes the blow-up morphism introduced in Section \ref{Section 5.2}.

Since $x\in E_{0}\times E_{1}$, it is not hard to see that $\BP^{1}_{x}$ is a closed subscheme of the exceptional divisor $E^{\prime}$. Without loss of generality, we use
\begin{center}
$[1:\Tilde{x}_{1}^{\prime}:\Tilde{x}_{2}^{\prime}]\times [1:\Tilde{x}_{4}^{\prime}:\Tilde{x}_{5}^{\prime}]$   
\end{center}
to denote the coordinates of $x$. By pluging the coordinates of $x$ into the Equation \eqref{Eqaution of Y'}, we get the coordinates 
\begin{center}
$[1:\Tilde{x}_{1}^{\prime}:\Tilde{x}_{2}^{\prime}]\times [1:\Tilde{x}_{4}^{\prime}:\Tilde{x}_{5}^{\prime}]\times[\Tilde{y}_{0}:\Tilde{y}_{1}]$ of $\BP^{1}_{x}$.    
\end{center}

Recall that the morphism  $p\colon\CZ\coloneqq\text{Bl}_{M^{G}}(M)\rightarrow \text{Bl}_{E_{0}\times E_{1}}(\BP^{2}\times \BP^{2})$ is finite, flat and of degree $3$. Using the isomorphism in the proof of Theorem \ref{Theorem 2}, we can compute the preimage  $p^{-1}(\BP^{1}_{x})$ of $\BP^{1}_{x}$, even explicitly in coordinates. 

Recall the isomorphism: $\Uxxyy\cong W_{0,3,1}$ given by 
\begin{equation}\label{coordinate map 1'}
(x_{1},x_{2},y_{4},y_{5},x_{3}^{\prime})\mapsto 
 (\Tilde{x}_{1},\Tilde{x}_{2},\Tilde{x}_{4},\Tilde{x}_{5},\Tilde{y}_{0}) 
\end{equation}
from the proof of Theorem \ref{Theorem 2}. If we restrict the morphism $p$ to the affine open subvariety $\Uxyy$ in $\text{Bl}_{M^{G}}(M)$, the image $p(\Uxyy)$ is exactly $\Uxxyy$ in $\text{Bl}_{E_{0}\times E_{1}}(\BP^{2}\times \BP^{2})$. Furthermore,
the associated map of $p$, restricted on $\Uxyy$, is
\begin{align*}
 \frac{\BC[x_{1},x_{2},x_{3}^{3},y_{4},y_{5}]}{(F_{0}(1,x_{1},x_{2})+x_{3}^{3}F_{1}(1,y_{4},y_{5}))}\longrightarrow 
 \frac{\BC[x_{1},x_{2},x_{3},y_{4},y_{5}]}{( F_{0}(1,x_{1},x_{2})+x_{3}^{3}F_{1}(1,y_{4},y_{5}))}.
 \end{align*}
 Therefore, the restriction of the morphism $p$ on $\Uxyy$ is given by
 \begin{equation}\label{coordinate map 2'}
 (x_{1},x_{2},y_{4},y_{5},x_{3})\mapsto (x_{1},x_{2},y_{4},y_{5},x_{3}^{\prime}\coloneqq x_{3}^{3}).
 \end{equation}
Here we use the notations of the proofs of Theorems \ref{Theorem 2}, \ref{Theorem 3}. 

Combining the maps \ref{coordinate map 1'} and \ref{coordinate map 2'}, yields:
\begin{equation*}
 (x_{1},x_{2},y_{4},y_{5},x_{3})\mapsto (\Tilde{x}_{1},\Tilde{x}_{2}, \Tilde{x}_{4},\Tilde{x}_{5},\Tilde{y}_{0}=x^{3}_{3})
\end{equation*}
This implies that the preimage $p^{-1}(\BP^{1}_{x})$ is isomorphic to $\BP^{1}$ and we are justified in denoting it by $\widetilde{\BP^{1}}$. The coordinates of $\widetilde{\BP^{1}}$ can be denoted by $[x_{0}:x_{3}]$, since all the other coordinates are fixed. Moreover, the morphism $p$, restricted on $\widetilde{\BP^{1}}$, is given by 
\begin{equation*}
p|_{\widetilde{\BP^{1}}}\coloneqq \widetilde{\BP^{1}}\rightarrow \BP^{1}_{x} \ \ [x_{0}:x_{3}]\mapsto [\Tilde{y}_{0}=x_{3}^{3},\Tilde{y}_{1}=x_{0}^{3}].
\end{equation*}

We know that the morphism $q\colon \CZ\rightarrow M$ is exactly the blow-up of the fixed locus $M^{G}$ and 
if we restrict the morphism $q$ to $\widetilde{\BP^{1}}$, then by a straightforward computation, we know that the restriction of the morphism $q$ to $\widetilde{\BP^{1}}$ is an isomorphism, which is given by
\begin{equation*}
 [x_{0}:x_{3}]\mapsto [x_{0}:x_{3}].
\end{equation*}
Since the coordinates of the image $q(\widetilde{\BP^{1}})$ of $\widetilde{\BP^{1}}$ is the same as the coordinates of $\widetilde{\BP^{1}}$, we will not distinguish them and it is clear that both carry the same $G$ action. Moreover, we also use $\widetilde{\BP^{1}}$ to denote $q(\widetilde{\BP^{1}})$.

Now, if we want to compute the image of $p_{*}(\textbf{L}q^{*}\CO_{M}(i)\otimes \chi_{j})|_{\BP^{1}_{x}}$, it is enough to compute $p_{*}(\textbf{L}q^{*}(\CO_{M}(i)\otimes \chi_{j}|_{\widetilde{\BP^{1}}}))$. And by the previous discussion, we know that the restriction of $q$ to $\widetilde{\BP^{1}}$ is an isomorphism. So the $G$-equivariant pullback $\textbf{L}q^{*}(\CO_{M}(i)\otimes \chi_{j}|_{\widetilde{\BP^{1}}})$ is exactly $\CO_{\CZ}(i)\otimes \chi_{j}|_{\widetilde{\BP^{1}}}$. Since the restriction of $\CO_{\CZ}(i)$ to $\widetilde{\BP^{1}}$ is $\CO_{\widetilde{\BP^{1}}}(i)$, we know that
\begin{center}
$q^{*}(\CO_{M}(i)\otimes \chi_{j}|_{\widetilde{\BP^{1}}})\cong \CO_{\widetilde{\BP^{1}}}(i)\otimes \chi_{j}$.   
\end{center}
Next we need to compute $p_{*}(\textbf{L}q^{*}\CO_{M}(i)\otimes \chi_{j})|_{\BP^{1}_{x}}$, which is equal to $p_{*}(\CO_{\widetilde{\BP^{1}}}(i)\otimes \chi_{j})$. We know that
$p_{*}(\textbf{L}q^{*}\CO_{M}(i)\otimes \chi_{j})$ is a line bundle  so $p_{*}(\textbf{L}q^{*}\CO_{M}(i)\otimes \chi_{j})$ restricted on $\BP^{1}_{x}$ is also a line bundle and it is equal to $\CO_{\BP^{1}_{x}}(-a_{i,j})$, since $E^{\prime}$ is an exceptional divisor. In order to compute the line bundle $\CO_{\BP^{1}_{x}}(-a_{i,j})$ on $\BP^{1}_{x}$ explicitly, it is enough to compute the dimensions of the sheaf cohomology
\begin{center}
$h^{l}(\BP_{x}^{1},p_{*}(\CO_{C}(i)\otimes \chi_{j}))$  
\end{center}
for $l=0,\,1$.
\begin{proposition}\label{Theorem about a}
With notation as in Definition \ref{definition about a,b,c}, we have
\begin{align*}
 & a_{0,0}=0,\ \  a_{0,1}=1,\ \ a_{0,2}=1,\\
 & a_{1,0}=0,\ \  a_{1,1}=1,\ \ a_{1,2}=0,\\
 & a_{2,0}=0,\ \  a_{2,1}=0,\ \ a_{2,2}=0.
\end{align*}    
\end{proposition}
\begin{proof}
Recall that $\textbf{R}p_{*}$ is the composition of functors:
\begin{center}
$D^{b}_{G}(\CZ)\underset{}{\overset{\textbf{R}p_{*}^{G}}{\longrightarrow}} D^{b}_{G}(Y)\underset{}{\overset{[\,\,]^{G}}{\longrightarrow}} D^{b}(Y)$.    
\end{center}
Hence, 
\begin{center}
$H^{l}(\BP^{1}_{x},p_{*}(\CO_{\widetilde{\BP^{1}}}(i)\otimes \chi_{j}))= H^{l}(\BP^{1}_{x},[p_{*}^{G}(\CO_{\Tilde{\BP^{1}}}(i)\otimes \chi_{j})]^{G})$ for $l=0,\,1$. 
\end{center}
for $l=0,\,1$.
By Proposition \ref{technical and important proposition} and \ref{important proposition 2}, we know
\begin{center}
 $h_{*}^{G}(-\otimes \chi)\cong h_{*}^{G}(-)\otimes \chi$ and
 $(-)^{G}\circ h_{*}^{G}\cong h_{*}^{G}\circ(-)^{G}$,
\end{center}
where $h_{*}^{G}\coloneqq \text{Coh}_{G}(Y)\rightarrow \text{Coh}_{G}(\text{Spec}(\BC))$ is the equivariant global sections functor.
Since the invariant functor $(-)^{G}$ and $-\otimes \chi\colon \text{Coh}_{G}(M)\rightarrow \text{Coh}_{G}(M)$ are exact, we have the following isomorphisms of equivariant derived functors:
\begin{center}
 $\textbf{R}h_{*}^{G}(-\otimes \chi)\cong \textbf{R}h_{*}^{G}(-)\otimes \chi$ and
 $(-)^{G}\circ \textbf{R}h_{*}^{G}\cong \textbf{R}h_{*}^{G}\circ(-)^{G}$.
\end{center}
The derived pushforward $\textbf({R}h_{*}^{G})^{l}$ is the sheaf cohomology $H^{l}$ for $l=0,1$.
Therefore, 
\begin{center}
$H^{l}(\BP^{1}_{x},[p_{*}^{G}(\CO_{\widetilde{\BP^{1}}}(i)\otimes \chi_{j})]^{G})\cong [H^{l}(\BP^{1}_{x}, p_{*}^{G}(\CO_{\widetilde{\BP^{1}}}(i))\otimes\chi_{j}]^{G}$, 
\end{center}
where we use $p_{*}^{G}(\CO_{\widetilde{\BP^{1}}}(i))$ by abuse of notation to denote $\text{Res} (p_{*}^{G}(\CO_{\widetilde{\BP^{1}}}(i))$.

Since the derived pushforward functor $p_{*}^{G}$ is exact, we have
\begin{center}
$H^{l}(\BP^{1}_{x}, p_{*}^{G}(\CO_{\widetilde{\BP^{1}}}(i))\cong H^{l}(\widetilde{\BP^{1}},\CO_{\widetilde{\BP^{1}}}(i))$.    
\end{center}
Now, to compute $h^{l}(\BP^{1}_{x},p_{*}(\CO_{\widetilde{\BP^{1}}}(i)\otimes \chi_{j}))$ for $l=1,2$, it is enough to compute the dimension $\text{dim}_{\BC}[H^{l}(\widetilde{\BP^{1}},\CO_{\Tilde{\BP^{1}}}(i))\otimes \chi_{j}]^{G}$.

Since we know that the coordinates of $\widetilde{\BP^{1}}$ are $[x_{0}:x_{3}]$, we can explicitly describe the natural $G$-structure of $H^{l}(\widetilde{\BP^{1}},\CO_{\widetilde{\BP^{1}}}(i))$:

$H^{1}(\widetilde{\BP^{1}},\CO_{\Tilde{\BP^{1}}}(i))=0$ for $i=0,1,2$, hence 
\begin{equation*}
\text{dim}_{\BC}[H^{1}(\widetilde{\BP^{1}},\CO_{\Tilde{\BP^{1}}}(i))\otimes \chi_{j}]^{G}=0.    
\end{equation*}
for $i,j=0,1,2$.

\noindent For $i=0,\,l=0$,
\begin{align*}
  & \text{dim}_{\BC}[H^{0}(\widetilde{\BP^{1}},\CO_{\Tilde{\BP^{1}}})\otimes \chi_{0}]^{G}= \text{dim}_{\BC}[\text{Span}_{\BC}(1\otimes \chi_{0})]^{G}=1   \\
  & \text{dim}_{\BC}[H^{0}(\widetilde{\BP^{1}},\CO_{\Tilde{\BP^{1}}})\otimes \chi_{1}]^{G}= \text{dim}_{\BC}[\text{Span}_{\BC}(1\otimes \chi_{1})]^{G}=0\\
  & \text{dim}_{\BC}[H^{0}(\widetilde{\BP^{1}},\CO_{\Tilde{\BP^{1}}})\otimes \chi_{2}]^{G}= \text{dim}_{\BC}[\text{Span}_{\BC}(1\otimes \chi_{2})]^{G}=0.
\end{align*}
Combining these results, we get
\begin{equation*}
a_{0,0}=0,\ \  a_{0,1}=1,\ \ a_{0,2}=1.   
\end{equation*}

\noindent For $i=1,\,l=0$,
\begin{align*}
  & \text{dim}_{\BC}[H^{0}(\widetilde{\BP^{1}},\CO_{\Tilde{\BP^{1}}}(1))\otimes \chi_{0}]^{G}= \text{dim}_{\BC}[\text{Span}_{\BC}(x_{0}\otimes \chi_{0},x_{3}\otimes \chi_{0})]^{G}=1\\
  & \text{dim}_{\BC}[H^{0}(\widetilde{\BP^{1}},\CO_{\Tilde{\BP^{1}}}(1))\otimes \chi_{1}]^{G}= \text{dim}_{\BC}[\text{Span}_{\BC}(x_{0}\otimes \chi_{1},x_{3}\otimes \chi_{1})]^{G}=0\\
  & \text{dim}_{\BC}[H^{0}(\widetilde{\BP^{1}},\CO_{\Tilde{\BP^{1}}}(1))\otimes \chi_{2}]^{G}= \text{dim}_{\BC}[\text{Span}_{\BC}(x_{0}\otimes \chi_{2},x_{3}\otimes \chi_{2})]^{G}=1.
\end{align*}
Thus
\begin{equation*}
a_{1,0}=0,\ \  a_{1,1}=1,\ \ a_{1,2}=0. 
\end{equation*}

\noindent \noindent For $i=2,\,l=0$,
\begin{align*}
  & \text{dim}_{\BC}[H^{0}(\widetilde{\BP^{1}},\CO_{\Tilde{\BP^{1}}}(2))\otimes \chi_{0}]^{G}= \text{dim}_{\BC}[\text{Span}_{\BC}(x_{0}^{2}\otimes \chi_{0},x_{3}^{2}\otimes \chi_{0}, x_{0}x_{3}\otimes \chi_{0})]^{G}= 1\\
  & \text{dim}_{\BC}[H^{0}(\widetilde{\BP^{1}},\CO_{\Tilde{\BP^{1}}}(2))\otimes \chi_{1}]^{G}= \text{dim}_{\BC}[\text{Span}_{\BC}(x_{0}^{2}\otimes \chi_{1},x_{3}^{2}\otimes \chi_{1}, x_{0}x_{3}\otimes \chi_{1})]^{G}= 1\\
  & \text{dim}_{\BC}[H^{0}(\widetilde{\BP^{1}},\CO_{\Tilde{\BP^{1}}}(2))\otimes \chi_{2}]^{G}=\text{dim}_{\BC}[\text{Span}_{\BC}(x_{0}^{2}\otimes \chi_{2},x_{3}^{2}\otimes \chi_{2}, x_{0}x_{3}\otimes \chi_{2})]^{G}= 1.
\end{align*}
Therefore,
\begin{equation*}
a_{2,0}=0,\ \  a_{2,1}=0,\ \ a_{2,2}=0.  
\end{equation*} 
\end{proof}
\subsubsection{The coefficients $b_{i,j}$}\label{subsection about b}
Without loss of generality, we take 
\begin{center}
$y=[1:0:0]\in \BP^{2}\setminus E_{1}$    
\end{center}
 from the second factor of $\BP^{2}$ of $Y\cong \text{Bl}_{E_{0}\times E_{1}}(\BP^{2}\times \BP^{2})$. Then we take a line $L_{y}\cong \BP^{1}$ in $\text{Bl}_{E_{0}\times E_{1}}(\BP^{2}\times \BP^{2})$, with coordinates
\begin{equation*}
\Tilde{x}_{2}=0,\ \Tilde{x}_{3}=1,\  \Tilde{x}_{4}=0,\ \Tilde{x}_{5}=0,\ 
\end{equation*}
where $[\Tilde{x}_{0}:\Tilde{x}_{1}:\Tilde{x}_{2}]\times[\Tilde{x}_{3}:\Tilde{x}_{4}:\Tilde{x}_{5}]$ are the coordinates of $\BP^{2}\times \BP^{2}$. Recall that the blow-up $\text{Bl}_{E_{0}\times E_{1}}(\BP^{2}\times \BP^{2})$ is cut out by the equation 
\begin{equation}\label{Eqaution of Y}
\Tilde{y}_{1}F_{0}(\Tilde{x}_{0},\Tilde{x}_{1},\Tilde{x}_{2})+\Tilde{y}_{0}F_{1}(\Tilde{x}_{3},\Tilde{x}_{4},\Tilde{x}_{5})=0   
\end{equation}
in $\BP^{2}\times \BP^{2}\times \BP^{1}$, where $[\Tilde{x}_{0}:\Tilde{x}_{1}:\Tilde{x}_{2}]\times[\Tilde{x}_{3}:\Tilde{x}_{4}:\Tilde{x}_{5}]\times [\Tilde{y}_{0}:\Tilde{y}_{1}]$ are the coordinates of $\BP^{2}\times \BP^{2}\times \BP^{1}$. By inserting the equation of $L_{y}$ in $\BP^{2}\times \BP^{2}$, the equation \eqref{Eqaution of Y} becomes
\begin{equation}
\Tilde{y}_{1}F_{0}(\Tilde{x}_{0},\Tilde{x}_{1},0)+\Tilde{y}_{0}F_{1}(1,0,0)=0.     
\end{equation}
Recall that the morphism  $p\colon\CZ\coloneqq\text{Bl}_{M^{G}}(M)\rightarrow \text{Bl}_{E_{0}\times E_{1}}(\BP^{2}\times \BP^{2})$ is finite, flat and of degree $3$. Using the isomorphism in the proof of Theorem \ref{Theorem 2}, we can compute the preimage  $p^{-1}(L_{y})$ of $L_{y}$, even explicitly in coordinates. 

Recall the isomorphism $\Uxxyy\cong W_{0,3,1}$ given by 
\begin{equation}\label{coordinate map 1}
(x_{1},x_{2},y_{4},y_{5},x_{3}^{\prime})\mapsto 
 (\Tilde{x}_{1},\Tilde{x}_{2},\Tilde{x}_{4},\Tilde{x}_{5},\Tilde{y}_{0}) 
\end{equation}
from the proof of Theorem \ref{Theorem 2}. If we restrict the morphism $p$ to the affine open subvariety $\Uxyy$ in $\text{Bl}_{M^{G}}(M)$, the image $p(\Uxyy)$ is exactly $\Uxxyy$ in $\text{Bl}_{E_{0}\times E_{1}}(\BP^{2}\times \BP^{2})$. Furthermore,
the associated map of $p$, restricted on $\Uxyy$, is
\begin{align*}
 \frac{\BC[x_{1},x_{2},x_{3}^{3},y_{4},y_{5}]}{(F_{0}(1,x_{1},x_{2})+x_{3}^{3}F_{1}(1,y_{4},y_{5}))}\longrightarrow 
 \frac{\BC[x_{1},x_{2},x_{3},y_{4},y_{5}]}{( F_{0}(1,x_{1},x_{2})+x_{3}^{3}F_{1}(1,y_{4},y_{5}))}.
 \end{align*}
 Therefore, the restriction of the morphism $p$ on $\Uxyy$ is given by
 \begin{equation}\label{coordinate map 2}
 (x_{1},x_{2},y_{4},y_{5},x_{3})\mapsto (x_{1},x_{2},y_{4},y_{5},x_{3}^{\prime}\coloneqq x_{3}^{3}).
 \end{equation}
Here we use the notations of the proofs of Theorems \ref{Theorem 2}, \ref{Theorem 3}. 

Combining the maps \ref{coordinate map 1} and \ref{coordinate map 2}, yields:
\begin{equation*}
 (x_{1},x_{2},y_{4},y_{5},x_{3})\mapsto (\Tilde{x}_{1},\Tilde{x}_{2}, \Tilde{x}_{4},\Tilde{x}_{5},\Tilde{y}_{0}=x^{3}_{3}).
\end{equation*}
It is straightforward to verify that the equation of the preimage $p^{-1}(L_{y})$ restricted to $\Uxyy$ is
\begin{equation*}
F_{0}(1,x_{1},0)+x_{3}^{3}F_{1}(1,0,0)=0.
\end{equation*}
Therefore, it is not hard to verify that the global equation of $p^{-1}(L_{y})$ in $\BLP$ is given by
\begin{equation*}
F_{0}(x_{0},x_{1},0)+ x_{3}^{3}F_{1}(1,0,0)=0.  
\end{equation*}
By taking the plane $L$ to be a plane $\BP^{2}\subset\BLP$ with  coordinates $[x_{0} : x_{1} : 0 : x_{3}: 0 : 0]$, we know that the preimage $p^{-1}(L_{y})\subset L$.
If we further assume that $F_{0}(x_{0},x_{1},0)=0$ has three distinct solutions, then by the Jacobian criterion, the preimage $p^{-1}(L_{y})$ is smooth. Since $p^{-1}(L_{y})$ is also a hyperplane in $L$ of degree $3$, it is a smooth elliptic curve by the adjunction formula and we use $C$ to denote it.

We know that the morphism $q\colon \CZ\rightarrow M$ is exactly the blow-up of the fixed locus $M^{G}$ and if we restrict the morphism $p$ to $C$, then by a direct computation, we see that the restriction of the morphism $p$ to $C$ is an isomorphism, which is given by
\begin{equation*}
    [x_{0}:x_{1}:x_{3}]\mapsto [x_{0}:x_{1}:x_{3}].
\end{equation*}
Since the coordinates of the image $q(C)$ of $C$ is the same as the coordinates of $C$ and it is also clear that both carry the same $G$ action, we will not distinguish them and also use $C$ to denote $p(C)$.

Similiar to \ref{Subsubsection 6.3.1}, in order to compute $p_{*}(\textbf{L}q^{*}(\CO_{M}(i)\otimes \chi_{j})|_{L_{y}}$, it is enough to compute $p_{*}(\textbf{L}q^{*}(\CO_{M}(i)\otimes \chi_{j}|_{C}))$. By the discussion above, we know that the restriction of the  morphism $q$ on $C$ is an isomorphism. The $G$-equivariant pullback $\textbf{L}q^{*}(\CO_{M}(i)\otimes \chi_{j}|_{C})$ is then $\CO_{C}(i)\otimes\chi_{j}$, where $\CO_{C}(i)\coloneqq \CO_{\BP^{2}}(i)|_{C}$. In total,
\begin{equation*}
p_{*}(\textbf{L}q^{*}\CO_{M}(i)\otimes \chi_{j})|_{L_{y}}\cong p_{*}(\CO_{C}(i)\otimes \chi_{j}). 
\end{equation*}

We know that
$p_{*}(\textbf{L}q^{*}\CO_{M}(i)\otimes \chi_{j})$ is a line bundle, so is $p_{*}(\textbf{L}q^{*}\CO_{M}(i)\otimes \chi_{j})|_{L_{y}}$. In order to compute it explicitly, it suffices to compute the dimensions of the sheaf cohomology
\begin{center}
$h^{l}(L_{y},p_{*}(\CO_{C}(i)\otimes \chi_{j}))$
\end{center}
for $l=0,\,1$.
\begin{proposition}\label{Theorem about b}
With notation as in Definition \ref{definition about a,b,c}, we have:
\begin{align*}
& b_{0,0}= 0,\ \ b_{0,1}= -2,\ \ b_{0,2}= -1\\
& b_{1,0}= 1,\ \ b_{1,1}= -1,\ \  b_{1,2}= \,0\\
& b_{2,0}= 2,\ \ b_{2,1}= 0,  \ \ b_{2,2}= 1.
\end{align*}    
\end{proposition}
\begin{proof}
By using the same argument as in proof of Theorem \ref{Theorem about a}, we obtain an isomorphism:
\begin{equation*}
H^{l}(L_{y},[p_{*}^{G}(\CO_{\widetilde{\BP^{1}}}(i)\otimes \chi_{j})]^{G})\cong [H^{l}(L_{y}, p_{*}^{G}(\CO_{C}(3i))\otimes\chi_{j}]^{G}.   
\end{equation*}
Moreover, we have 
\begin{equation*}
H^{l}(L_{y}, p_{*}^{G}(\CO_{C}(i))\cong H^{l}(C,p_{*}(\CO_{C}(i)\otimes \chi_{j})).    
\end{equation*}

Therefore, in order to compute the dimension $h^{l}(C,p_{*}(\CO_{C}(3i)\otimes \chi_{j}))$ for $l=1,2$, it is enough to compute $\text{dim}_{\BC}[H^{l}(C,\CO_{C}(3i))\otimes \chi_{j}]^{G}$. By the following short exact sequence:
\begin{equation}\label{short exact sequence}
\begin{tikzcd}
0 \arrow[r] & \CO_{\BP^{2}}(-3) \arrow[r] & \CO_{\BP^{2}} \arrow[r] & \CO_{C} \arrow[r] & 0.
\end{tikzcd}       
\end{equation}
We have
\begin{align*}
& H^{0}(C,\CO_{C})\cong H^{0}(\BP^{2},\CO_{\BP^{2}})\cong\BC \chi_{0} \\
& H^1(C,\CO_{C})\cong H^2(\BP^{2},\CO_{\BP^{2}}(-3))= \BC\frac{1}{x_{0}x_{1}x_{3}}= \BC\chi_{2}.
\end{align*}
Then for $i=0,\,l=0,\,1$, we have 
\begin{align*}
& \text{dim}_{\BC}[H^{0}(C,\CO_{C})\otimes\chi_{0}]= \text{dim}_{\BC}[\text{Span}_{\BC}(\chi_{0})]^{G}=1,\\
& \text{dim}_{\BC}[H^{0}(C,\CO_{C})\otimes\chi_{1}]= \text{dim}_{\BC}[\text{Span}_{\BC}(\chi_{1})]^{G}=0,\\
& \text{dim}_{\BC}[H^{0}(C,\CO_{C})\otimes\chi_{2}]= \text{dim}_{\BC}[\text{Span}_{\BC}(\chi_{2})]^{G}=0,\\
& \text{dim}_{\BC}[H^{1}(C,\CO_{C})\otimes\chi_{0}]= \text{dim}_{\BC}[\text{Span}_{\BC}(\chi_{2})]^{G}=0,\\
& \text{dim}_{\BC}[H^{1}(C,\CO_{C})\otimes\chi_{1}]= \text{dim}_{\BC}[\text{Span}_{\BC}(\chi_{0})]^{G}=1,\\
& \text{dim}_{\BC}[H^{1}(C,\CO_{C})\otimes\chi_{2}]= \text{dim}_{\BC}[\text{Span}_{\BC}(\chi_{1})]^{G}=0.
\end{align*}

Hence,
\begin{center}
$b_{0,0}=0, \,b_{0,1}=-2, \,b_{0,2}=-1$.    
\end{center}

Now, we tensor \eqref{short exact sequence} with $\CO_{\BP^{2}}(1)$ to obtain the short exact sequence
\begin{center}
\begin{tikzcd}
0 \arrow[r] & \CO_{\BP^{2}}(-2) \arrow[r] & \CO_{\BP^{2}}(1) \arrow[r] & \CO_{C}(1) \arrow[r] & 0.
\end{tikzcd}   
\end{center}
We get that
\begin{align*}
  & H^{0}(C,\CO_{C}(1))\cong  H^{0}(\BP^{2},\CO_{\BP^{2}}(1))\cong\text{Span}_{\BC}(x_{0},x_{1},x_{3}) = \BC^{2}\chi_{0}\bigoplus\BC\chi_{1},\\
  & H^{1}(C,\CO_{C}(1))\cong H^{1}(\BP^{2},\CO_{\BP^{2}}(-2))=0.
\end{align*}
Then for $i=1,\,l=0,\,1$, 
\begin{align*}
& \text{dim}_{\BC}[H^{0}(C,\CO_{C}(1))\otimes\chi_{0}]= \text{dim}_{\BC}[(\BC^{2}\chi_{0}\bigoplus\BC\chi_{1})\otimes\chi_{0}]^{G}=2,\\
& \text{dim}_{\BC}[H^{0}(C,\CO_{C}(1))\otimes\chi_{1}]= \text{dim}_{\BC}[(\BC^{2}\chi_{0}\bigoplus\BC\chi_{1})\otimes\chi_{1}]^{G}=0,\\
& \text{dim}_{\BC}[H^{0}(C,\CO_{C}(1))\otimes\chi_{2}]= \text{dim}_{\BC}[(\BC^{2}\chi_{0}\bigoplus\BC\chi_{1})\otimes\chi_{2}]^{G}=1,\\
& \text{dim}_{\BC}[H^{1}(C,\CO_{C}(1))\otimes\chi_{j}]= \text{dim}_{\BC}[\text{Span}_{\BC}(0\otimes\chi_{j})]^{G}=0 \text{ for }j=0,1,2.
\end{align*}
Therefore,
\begin{equation*}
b_{1,0}=1,\, b_{1,1}=-1,\, b_{1,2}=0.
\end{equation*}

For $i=2$, we tensor the last short exact sequence further with $\CO_{\BP^{2}}(1)$ to get
\begin{center}
\begin{tikzcd}
0 \arrow[r] & \CO_{\BP^{2}}(-1) \arrow[r] & \CO_{\BP^{2}}(2) \arrow[r] & \CO_{C}(2) \arrow[r] & 0.
\end{tikzcd}   
\end{center}
Again, we get that
\begin{align*}
 H^{0}(\BP^{2},\CO_{C_{1}}(2))
 & \cong H^{0}(\BP^{2},\CO_{\BP^{2}}(2))
 \cong \text{ Span}_{\BC}(x_{0}^{2},x_{0}x_{1},x_{1}^{2},x_{0}x_{3},x_{1}x_{3},x_{3}^{2})\\ & = \BC^{3}\chi_{0}\bigoplus\BC^{2}\chi_{1}\bigoplus \BC\chi_{2},\\
 H^{1}(C,\CO_{C}(2))& \cong H^{2}(\BP^{2},\CO_{\BP^{2}}(-1))=0.
\end{align*}
For $i=2$ and $l=0,\,1$,
\begin{align*}
& \text{dim}_{\BC}[H^{0}(C,\CO_{C}(2))\otimes\chi_{0}]= \text{dim}_{\BC}[(\BC^{3}\chi_{0}\bigoplus\BC^{2}\chi_{1}\bigoplus \BC\chi_{2})\otimes\chi_{0}]^{G}=3,\\
& \text{dim}_{\BC}[H^{0}(C,\CO_{C}(2))\otimes\chi_{1}]= \text{dim}_{\BC}[(\BC^{3}\chi_{0}\bigoplus\BC^{2}\chi_{1}\bigoplus \BC\chi_{2})\otimes\chi_{1}]^{G}=1,\\
& \text{dim}_{\BC}[H^{0}(C,\CO_{C}(2))\otimes\chi_{2}]= \text{dim}_{\BC}[(\BC^{3}\chi_{0}\bigoplus\BC^{2}\chi_{1}\bigoplus \BC\chi_{2})\otimes\chi_{2}]^{G}=2,\\
& \text{dim}_{\BC}[H^{1}(C,\CO_{C}(2))\otimes\chi_{j}]= \text{dim}_{\BC}[\text{Span}_{\BC}(0\otimes\chi_{j})]^{G}=0 \text{ for }j=0,1,2.
\end{align*}
Finally, 
\begin{equation*}
  b_{2,0}=2,\  b_{2,1}=0,\ b_{2,2}= 1.
\end{equation*}
\end{proof}
\subsubsection{The coefficients $c_{i,j}$}\label{Subsection about c}
By renaming $E_{1}$ to $E_{0}$, we can easily repeat every step in the previous section to compute the coefficients $c_{i,j}$. We omit the details and only state the result.
\begin{proposition}\label{theorem about c}
With notation as in Definition \ref{definition about a,b,c}, we have
\begin{align*}
& c_{0,0}=0,\ c_{0,1}=-1,\ c_{0,2}=-2,\\  
& c_{1,0}=0,\ c_{1,1}=-1,\ c_{1,2}=1, \\
& c_{2,0}=0,\ c_{2,1}= 2,\ c_{2,2}=1.
\end{align*}    
\end{proposition}
Combining Theorem \ref{Theorem about a}, \ref{Theorem about b} and \ref{theorem about c}, we get all line bundles $\Psi(\CO_{M}(i)\otimes \chi_{j})$ in terms of divisors:

\begin{table}[H]\label{The table}
\centering
\begin{tabular}{c c c c} 
\hline\hline 
Line bundles & Divisor $E$ & Divisor $H_{1}$ & Divisors $H_{2}$ \\ [0.5ex] 
\hline 
$\Psi(\CO_{M}\otimes \chi_{0})$ & 0 & 0 & 0 \\ 
$\Psi(\CO_{M}\otimes \chi_{1})$ & 1 & -2 & -1 \\
$\Psi(\CO_{M}\otimes \chi_{2})$ & 1 & -1 & -2 \\[1ex]
\hline
$\Psi(\CO_{M}(1)\otimes \chi_{0})$ & 0 & 1 & 0 \\
$\Psi(\CO_{M}(1)\otimes \chi_{1})$ & 1 & -1 & -1 \\
$\Psi(\CO_{M}(1)\otimes \chi_{2})$ & 0 & 0 & 1 \\[1ex]
\hline
$\Psi(\CO_{M}(2)\otimes \chi_{0})$ & 0 & 2 & 0 \\
$\Psi(\CO_{M}(2)\otimes \chi_{1})$ & 0 & 0 & 2 \\
$\Psi(\CO_{M}(2)\otimes \chi_{2})$ & 0 & 1 & 1 \\[1ex] 
\hline 
\end{tabular}
\label{table:nonlin} 
\end{table}
\subsection{The isomorphism between $D^{b}(E_{0}\times E_{1})$ and $Ku_{G}(M)$}
In the last section, we calculated the images $\Psi(\CO_{M}(i)\otimes \chi_{j})$ for $0\leq i\leq 2$ and $0\leq j\leq 2$, which are all line bundles on $Y$. Since $\Psi$ is an equivalence between $D^{b}_{G}(M)$ and $D^{b}(Y)$, we have the following decomposition for $D^{b}(Y)$:
\begin{equation}\label{Equation 19}
\begin{split}
& D^{b}(Y)=  \langle \Psi(\mathcal{K}u_{G}(M)), \CO_{Y},\CO_{Y}(1,-2,-1),\CO_{Y}(1,-1,-2), \\ & \CO_{Y}(0,1,0),
\CO_{Y}(1,-1,-1),\CO_{Y}(0,0,1),\CO_{Y}(0,2,0),\CO_{Y}(0,0,2), \\
& \CO_{Y}(0,1,1)\rangle.
\end{split}
\end{equation}
Also recall the semiorthogonal decomposition of $D^{b}(Y)$ from \eqref{equation 12}:
\begin{equation*}
    \begin{split}
       D^{b}(Y)= & \langle D_{-1},\CO_{Y}(0,0,0),\CO_{Y}(0,0,1),
       \CO_{Y}(0,0,2),
       \CO_{Y}
       (0,1,0),
      \CO_{Y}(0,1,1),
      \\&\CO_{Y}(0,1,2),\CO_{Y}(0,2,0),\CO_{Y}(0,2, 1),\CO_{Y}(0,2,2)\rangle.
    \end{split}
\end{equation*}
Now, the idea is to find a functor to give isomorphisms from the nine line bundles in \eqref{Equation 19} to the nine line bundles above. We will do this by twisting $\Psi$ of Section \ref{Section about overview of Step 2} by mutation functors
\begin{lemma}\label{Lemma 6.9}

$\textbf{R}\text{Hom}(\CO_{Y},\CO_{Y}(1,-1,-1))\cong 0$.
\end{lemma}
\begin{proof}
Since 
\begin{center}
$\Psi(\CO_{M}\otimes \chi_{0})\cong \CO_{Y}$, $\Psi(\CO_{M}(1)\otimes \chi_{1})\cong \CO_{Y}(1,-1,-1)$,  
\end{center}
and $\Psi$ is an equivalence, we have
\begin{center}
$\text{Hom}_{D^{b}(Y)}(\CO_{Y},\CO_{Y}(1,-1,-1))\cong\text{Hom}_{D^{b}_{G}(M)}(\CO_{M}$ $\otimes \chi_{0}, \CO_{M}(1)\otimes \chi_{1})$.    
\end{center}
Using the short exact sequence,
\begin{center}
\begin{tikzcd}
0 \arrow[r] & \CO_{\BP^{5}}(-2) \arrow[r] & \CO_{\BP^{5}}(1) \arrow[r] & \CO_{M}(1) \arrow[r] & 0,
\end{tikzcd}     
\end{center}
we get
\begin{equation*}
    H^{i}(M,\CO_{M}(1)) = 
    \begin{cases}
             \BC^{3}\chi_{0}\bigoplus\BC^{3}\chi_{1} &\text{if $i=0$ }\\
                        0 &    \text{if $i> 0$}.
                    \end{cases}
\end{equation*}
We conclude that 
\begin{center}
$\text{Hom}_{G}^{i}(\CO_{M}$ $\otimes \chi_{0}, \CO_{M}(1)\otimes \chi_{1})\cong [H^{i}(M,\CO_{M}(1))\otimes \chi_{1}]^{G}\cong 0$    
\end{center}
for all $i\geq 0$.
\end{proof}
Now we are equipped to prove the main theorem of this paper.
\begin{theorem}\label{main theorem}

$\mathcal{K}u_{G}(M)\cong D^{b}(E_{0}\times E_{1})$.
\end{theorem}
\begin{proof}
First, consider the inverse Serre functor $S^{-1}(-)\coloneqq (-)\otimes \omega_{Y}^{-1}[-4]$ on $D^{b}(Y)$, where $\omega_{Y}$ is the canonical line bundle of $Y$. We use the functor $S^{-1}[4](-)\coloneqq(-)\otimes \omega_{Y}$ instead of $S^{-1}(-)$ for simplicity. And $\omega_{Y}^{-1}\cong \CO_{Y}(-1,3,3)$, since $\omega_{Y}=(p^{\prime})^{*}\omega_{\BP^{2}\times \BP^{2}}\otimes \CO_{Y}(E)$, where $p'$ is the projection map from $Y$ to $\BP^{2}\times\BP^{2}$ and $E$ is the exceptional divisor.

We first apply the functor $S^{-1}[4]$ and get the following semiorthogonal decomposition:
\begin{equation*}
\begin{split}
D^{b}(Y)= & \langle \CO_{Y},\CO_{Y}(1,-2,-1),\CO_{Y}(1,-1,-2),\CO_{Y}(0,1,0),
\CO_{Y}(1,-1,-1),\\
& \CO_{Y}(0,0,1),\CO_{Y}(0,2,0),\CO_{Y}(0,0,2), \CO_{Y}(0,1,1),\Psi(\mathcal{K}u_{G}(M))\otimes \omega_{Y}^{-1}\rangle.
\end{split}
\end{equation*}
By Proposition \ref{KP16}, since $\langle \Psi(\CO_{M}(i)\otimes \chi_{0}),\Psi(\CO_{M}(i)\otimes \chi_{1}),\Psi(\CO_{M}(i)\otimes \chi_{2})\rangle$ is a completely orthogonal decomposition for all $0\leq i\leq 2$,
the decompositions:
\begin{align*}
& \langle \CO_{Y},\CO_{Y}(1,-2,-1),\CO_{Y}(1,-1,-2) \rangle \\
& \langle \CO_{Y}(0,1,0),\CO_{Y}(1,-1,-1),\CO_{Y}(0,0,1)\rangle\\
& \langle \CO_{Y}(0,2,0),\CO_{Y}(0,0,2), \CO_{Y}(0,1,1) \rangle
\end{align*}
are also completely orthogonal.
As a result, we can change the order of the above semiorthogonal decomposition of $D^{b}(Y)$ to 
\begin{equation*}
\begin{split}
D^{b}(Y)= & \langle \CO_{Y}(1,-2,-1),\CO_{Y}(1,-1,-2),\CO_{Y},\CO_{Y}(1,-1,-1),\CO_{Y}(0,1,0),\\
& \CO_{Y}(0,0,1),\CO_{Y}(0,2,0),\CO_{Y}(0,0,2), \CO_{Y}(0,1,1),\Psi(\mathcal{K}u_{G}(M))\otimes \omega_{Y}^{-1}\rangle.
\end{split}
\end{equation*}
Next, we use the functor $S^{-1}[4]$ again to get
\begin{equation*}
\begin{split}
D^{b}(Y)= & \langle \CO_{Y},\CO_{Y}(1,-1,-1),\CO_{Y}(0,1,0),
 \CO_{Y}(0,0,1),\CO_{Y}(0,2,0),
 \CO_{Y}(0,0,2),\\
 & \CO_{Y}(0,1,1),\Psi(\mathcal{K}u_{G}(M))\otimes \omega_{Y}^{-1},\CO_{Y}(0,2,1),\CO_{Y}(0,1,2)\rangle.
\end{split}
\end{equation*}
By Lemma \ref{Lemma 6.9}, we can change the order $\langle \CO_{Y},\CO_{Y}(1,-1,-1)\rangle$ to $\langle \CO_{Y}(1,-1,$ $-1),\CO_{Y}\rangle$. After that, one more application of $S^{-1}[n]$ derives
\begin{equation*}
\begin{split}
D^{b}(Y)= & \langle \CO_{Y},\CO_{Y}(0,1,0),
 \CO_{Y}(0,0,1),\CO_{Y}(0,2,0),
 \CO_{Y}(0,0,2),
  \CO_{Y}(0,1,1),\\
  &\Psi(\mathcal{K}u_{G}(M))\otimes \omega_{Y}^{-1},\CO_{Y}(0,2,1),\CO_{Y}(0,1,2),\CO_{Y}(0,2,2)\rangle.
\end{split}
\end{equation*}
Now, we use $\CA$ to denote the admissible subcategory 
\begin{center}
$\langle \CO_{Y},\CO_{Y}(0,1,0),
 \CO_{Y}(0,0,1),$ $\CO_{Y}(0,2,0),
 \CO_{Y}(0,0,2),
  \CO_{Y}(0,1,1)\rangle$.
\end{center}
 By using the mutation functor $\CL_{\CA}$, we get the following semiorthogonal decomposition of $D^{b}(Y)$:
\begin{equation*}\label{Last equation}
\begin{split}
D^{b}(Y)= & \langle \CL_{\CA}(\Psi(\mathcal{K}u_{G}(M))\otimes \omega_{Y}^{-1}),\CO_{Y},\CO_{Y}(0,1,0),
 \CO_{Y}(0,0,1),
 \CO_{Y}(0,2,0),\\
 &\CO_{Y}(0,0,2),
  \CO_{Y}(0,1,1),
  \CO_{Y}(0,2,1),\CO_{Y}(0,1,2),\CO_{Y}(0,2,2)\rangle
\end{split}
\end{equation*} 
Note that the semiorthogonal decomposition given by the last nine terms is excatly $D_{0}$ in \eqref{equation 12}, so we have
\begin{center}
$D_{0}^{\bot}\cong \CL_{\CA}(\Psi(Ku_{G}(M))\otimes \omega_{Y}^{-1})\cong \Psi(Ku_{G}(M))$.   
\end{center}
Then a combination of this and the isomorphism
\begin{center}
 $D^{b}(E_{0}\times E_{1})\cong D_{-1}\cong\,D_{0}^{\bot}$ 
\end{center}
explained at the beginning of Section \ref{Section 5.2} implies that 
\begin{center}
$\mathcal{K}u_{G}(M)\cong D^{b}(E_{0}\times E_{1})$.    
\end{center}
\end{proof}

\printbibliography
\end{document}